\documentclass[a4paper]{amsart}
\usepackage[T1]{fontenc}
\usepackage{lmodern}

\usepackage{amsmath,amsfonts,amsthm,amssymb,amscd, verbatim,graphicx,color,multirow,booktabs, bm,chngcntr,adjustbox,longtable,mathtools,microtype, xfrac, mathdots}
\usepackage[top = 78pt, bottom = 72pt, inner=90pt, outer=90pt]{geometry}
\usepackage{hyperref}
\usepackage[dvipsnames, table]{xcolor}
\hypersetup{colorlinks=true,linkcolor=Maroon, citecolor=OliveGreen}
\usepackage{tikz, tikz-cd, quiver}

\usepackage[capitalise,noabbrev]{cleveref}
\usepackage[shortlabels]{enumitem}

\usepackage{todonotes}

\definecolor{lavander}{HTML}{B56BEE}

\newcounter{proofeq}[section]
\renewcommand{\theproofeq}{\thesection.\arabic{proofeq}}

\newenvironment{proofequation}
  {\refstepcounter{proofeq}\begin{equation*}\tag{\theproofeq}}
  {\end{equation*}}

\renewcommand{\wr}{ \mathop{\mathrm{wr}} }
\newcommand{\N}{\mathbb{N}}

\newcommand{\A}{\mathsf{A}_{k}}
\newcommand{\D}{\mathsf{\Delta}_{h}}
\newcommand{\Ainf}{\mathsf{A}}

\newcommand{\B}{\mathbf{B}}
\newcommand{\Z}{\mathbb{Z}}
\newcommand{\NN}{\mathbb{N}}
\newcommand{\E}{\mathbb{E}}

\newcommand{\Aut}{\mathrm{Aut}}
\renewcommand{\P}{\mathbb{P}}

\newcommand{\ZZ}{\mathbb{Z}}

\newcommand{\1}{\mathbf{1}}
\newcommand{\Pa}{\mathcal{P}}
\newcommand{\cp}{\mathrm{cp}}
\newcommand{\sq}{\mathrm{sq}}

\newtheorem{thm}{Theorem}[]
\newtheorem{lemma}[thm]{Lemma}

\newtheorem{thmx}{Theorem}[]

\newtheorem{corx}[thmx]{Corollary}

\theoremstyle{definition}
\newtheorem{definition}[thm]{Definition}
\newtheorem{remark}[thm]{Remark}

\newtheorem{question}[thmx]{Question}
\newtheorem{example}[thm]{Example}
\newtheorem{conjecture}[thmx]{Conjecture}

\title{On the expected value of energy in groups}

\author[M.~Barbieri]{Marco Barbieri}
\address{Faculty of Mathematics and Physics, University of Ljubljana, Jadranska ulica 21, 1000 Ljubljana, Slovenia.} 
\email{marco.barbieri@fmf.uni-lj.si}

\author[M. ~Lekše ]{Maruša Lekše}
\address{Institute of Mathematics, Physics, and Mechanics, Jadranska ulica 19, 1000 Ljubljana, Slovenia. Also affiliated with: Faculty of Mathematics and Physics, University of Ljubljana, Jadranska ulica 21, 1000 Ljubljana, Slovenia.} 
\email{marusa.lekse@imfm.si}

\author[A.~Zozaya]{Andoni Zozaya} 
\address{Department of Statistics, Computer Science and Mathematics, Public University of Navarre, Campus of Arrosadia, 31006 Pamplona, Spain. Also affiliated with: INAMAT$^{2}$, Campus of Arrosadia, 31006 Pamplona, Spain}
\email{andoni.zozaya@unavarra.es}


\thanks{The research presented in this paper was initiated during a research visit by AZ to IMFM, funded by Slovenian Research and Innovation Agency, research programme number P1-0294. He gratefully acknowledges the support and hospitality provided. MB is supported by the Slovenian Research and Innovation Agency programme P1-0222 and grant J1-50001, and he is a member of the Italian GNSAGA INdAM research group. ML is supported by the Slovenian Research and Innovation Agency, programme number P1-0294.} 

\begin{document}	
	
    \begin{abstract}
        We obtain explicit upper and lower bounds for the expected action energy associated
with a pair $({\sf A},{\sf \Delta})$ of subsets sampled uniformly at random from a
permutation group and its domain, respectively. We then specialize these bounds to
multiplicative energy in several settings. In particular, we derive sharp asymptotic
formulae for the expected energy of pairs of the form $({\sf A},{\sf A})$ and
$({\sf A},{\sf A}^{-1})$. Finally, we apply these estimates to derive probabilistic
results on the existence of subsets with large growth and to compare the typical
behaviour of the cardinalities of the sets $|{\sf A}^{\ast 2}|$ and
$|{\sf A}{\sf A}^{-1}|$.
    \end{abstract}
	
    \maketitle
    \tableofcontents
    
\section{Introduction}
Sets $A$ whose sumsets $A+A$ or product sets $A^{\ast 2}$ have cardinality linear in
$|A|$ play a central role in additive combinatorics and group theory. This property is
closely tied to a strong algebraic structure, as exemplified by Freiman-type theorems and
the theory of approximate groups. A quantitative way to detect and measure this
phenomenon is provided by the \emph{energy}, which records the number of algebraic
coincidences in products (or, in the permutational representation context, in images)
among pairs of elements from two sets. (The definition of energy and a discussion of its basic
properties can be found in \cref{sec:energy}.)

A guiding question underlying this paper is the following: \emph{how uncommon is
structured behaviour among subsets sampled uniformly at random?} Rather than
focusing on extremal or explicitly constructed examples, we investigate the typical
behaviour of random subsets by computing the expected values of suitable energy
quantities. This probabilistic viewpoint allows us to place structural results into a
quantitative framework and to assess how rare certain additive-combinatorial phenomena
actually are in a generic group.

\subsection{Motivation}

Large energy forces many algebraic coincidences, while small energy typically
corresponds to expansion. Hence, energy serves as an intermediary between algebraic
structure and growth. This relationship is formalized by the inequality
\[
|AB| \ge \frac{|A|^2|B|^2}{E(A,B)} \,,
\]
where $A$ and $B$ are subsets of an ambient group~$G$, and~$E(A,B)$ is their energy. (The proof of this inequality, in a slightly more
general setting, can be found in \cref{lemma:CauchySchwartz}.) Upon fixing two positive
integers $k$ and $h$, and appropriately defining a way to sample a $k$-subset ${\sf A}$
and an $h$-subset ${\sf B}$ (which are now random variables taking values in the power
set of~$G$, see \cref{sec:probability}), Jensen's inequality yields
\[
\E[|{\sf A}{\sf B}|] \ge \frac{k^2h^2}{\E[E({\sf A}, \, {\sf B})]} \,.
\]
Thus, information on the expected value of energy immediately translates into lower
bounds for the size of product sets.

Let $\A$ be a $k$-subset of the integer interval $[-n,n]$ sampled uniformly at random.
Following \cite[Section~2.1]{Tointon}, the expected additive energy of
$(\A,\A)$ can be bounded from above by an explicit quantity, which in turn yields the lower
bound on the expected doubling
\begin{equation}\label{eq:intro-tao}
\liminf_n \mathbb{E}\bigl[\,|\A + \A|\,\bigr]
\ge
\frac{k^3}{2k+1} \,.
\end{equation}
This inequality shows that sets with unusually small sumsets are highly atypical among
random subsets of the integers.

The purpose of the present paper is to extend this philosophy far beyond the integer
setting. In fact, our aim is to extend \cref{eq:intro-tao} to a broader context and, as a secondary outcome, to study how
energies depend on algebraic properties of the ambient group.

\subsection{General bounds}

The main innovation used to achieve this goal consists of a double-counting technique
that translates the computation of the expected value of the energy of a uniformly
sampled $k$-subset into counting the number of $k$-subsets containing four specified
elements. More precisely, we endow a discrete group~$G$ with a finite filtration series $(F_n)$, choose a nondecreasing function
$k\colon\mathbb{N}\rightarrow\mathbb{N}$, and sample uniformly at random $k$-subsets $\A$ of
$F_n$. If $G$ is a permutation group, we also fix a filtration of the domain
$(\Phi_n)$ and a nondecreasing function
$h\colon\mathbb{N}\to\mathbb{N}$, and proceed analogously by sampling $h$-subsets $\D$
of $\Phi_n$. (The underlying probability spaces are defined in
\cref{sec:probability}.)

This framework allows us to determine upper and lower bounds for the action energy of
a pair $(\A,\D)$, where $\A$ is a $k$-subset of a permutation group~$G$ and $\D$ is an
$h$-subset of points. The bounds depend only on $k$, $h$, and the size of the acting
domain, and are valid for arbitrary actions of discrete groups (see
\cref{thm:action}). We emphasize that, if $k$ and $h$ are both bounded functions while $G$ and its
domain are infinite, the expected energy converges to the lower bound $\lim_n kh$. This value
coincides with the energy of the generating set of a free group acting on a set of
vertices of the infinite tree that are pairwise at distance at least~$2$ (see
\cref{example:free}). Furthermore, in the regime $kh \ll |\Phi_n|$, combining
\cref{thm:action} with \cref{lemma:CauchySchwartz} and Jensen's inequality, we promptly
obtain
\[
\E [ |\D \cdot \A|] \ge \frac{kh}{1 + o(1)} \,.
\]
Note that this is asymptotically best possible, as the image of a $h$-subset of points under the action of $k$ group elements cannot exceed $kh$.

Since the bounds of \cref{thm:action} are sharp and are attained by regular actions (the
lower bound is met by infinite groups with $k$ bounded, while the upper bound is met
by finite groups with $k$ equal to their cardinality), our next step is to specialize
the set $\D$ to a distinguished subset of~$G$, viewed as a permutation group on itself via its (right)
regular action.

Our first choice is $\D=\A$ itself, since upper bounds on $E(\A,\A)$ translate into
lower bounds on the cardinality of $\A^{\ast 2}$, in the spirit of
\cref{eq:intro-tao}. \cref{thm:multiplicative} provides upper and lower bounds for the
expected multiplicative energy of $\A$ with itself. These bounds are asymptotically
sharp in the regime $k \ll \sqrt{|F_n|}$, as illustrated by
\Cref{example:int,example: heissenberg}. They are expressed in terms of the size of the
intersection of the filtration with centralizers and are therefore sensitive to the
ambient group structure, revealing how non-commutativity influences typical energy
values.

Our second choice is $\D=\A^{-1}$, the set of inverses of elements of~$\A$.
\cref{thm:multiplicativeInverse} provides sharp upper and lower bounds for the expected
energy of the pair $(\A,\A^{-1})$, expressed in terms of the density of involutions in
the group. 
These bounds are typically tighter and exhibit cancellation phenomena that do not occur in \cref{thm:multiplicative}. 
As a consequence, the constraints on the cardinality $|\A\A^{-1}|$ are usually lower than those for $|\A^{\ast 2}|$.

Note that applying the usual formula to the upper bounds obtained in
Theorems~\ref{thm:multiplicative} and~\ref{thm:multiplicativeInverse} and substituting
the parameters corresponding to the worst-case upper bounds for the energies yields, in the regime $k \ll \sqrt{|F_n|}$,
\[
\E[|\A^{\ast 2}|] \ge \frac{k^2}{3+o(1)},
\quad \hbox{and} \quad 
\E[|\A\A^{-1}|] \ge \frac{k^2}{3+o(1)} \,.
\]

\subsection{Finite groups}

In the finite group setting, Theorems~\ref{thm:multiplicative} and
\ref{thm:multiplicativeInverse} can be refined using character theory.

\cref{cor:multiplicativeFinite} gives an explicit closed formula for the multiplicative
energy of the pair $(\A,\A)$ in terms of the number of conjugacy classes and the number
of real and quaternionic irreducible characters of the group (these notions will be
defined in \cref{sec:char}). For $k$ being a fixed integer and for a sequence of groups whose order tends to infinity, this already yields precise asymptotics for large classes of groups, including
nonabelian simple groups. Indeed, in this case, \cref{eq:simple},
\cref{lemma:CauchySchwartz}, and Jensen's inequality imply
\[
\E[|\A^{\ast 2}|] \ge \frac{k^2}{1+o(1)} \,.
\]
This behaviour is the best possible, since $|\A^{\ast 2}| \le k^2$. (Furthermore, we can use
\cref{eq:simple} to obtain the precise decay of the $o(1)$ term in the previous
inequality.)

\cref{cor:multiplicativeInverseFinite} specialises this formula to the multiplicative
energy of the pair $(\A,\A^{-1})$. The resulting expression is again explicit and
reduces the computation of expected energy to character-theoretic data determining
the number of involutions (see \cref{lemma:characterMagic}). Moreover, as the asymptotic behaviour of $\E[E(\A,\A^{-1})]$ only depends on the proportion of involutions in the groups under consideration, we get the same asymptotic expansion for the two widely different classes of solvable groups of odd order and of simple groups.
Furthermore, if the \emph{limes superior} of the proportion of involutions is bounded by a real number $\iota$, we obtain
\[
\E[|\A\A^{-1}|] \ge \frac{k^2}{2+ \iota +o(1)} \,.
\]

\subsection{Infinite groups}
In the infinite group setting, refinements of Theorems~\ref{thm:multiplicative} and \ref{thm:multiplicativeInverse} are more delicate. We begin with the pair $(\A,\A)$. The finite case already suggests that the leading term of the asymptotics should involve  the commuting probability $\cp_\mu(G)$ and the square-incidence probability $\sq_\mu(G)$ of the group. (Here $\cp_\mu(G)$ is the probability that two random elements commute, while $\sq_\mu(G)$ is the probability that two random elements share the same square.) These invariants are defined respectively in Equations \eqref{eq: commuting proba} and \eqref{eq: sq}. As the notation indicates, they might \emph{a priori} depend on the choice of measure, or, equivalently, on the filtration series (see \cref{subsec:infinite}).
More precisely, \cref{cor:multiplicativeGroupMeasure} shows that for an infinite discrete group $G$,  a finite filtration series $(F_n)$, and $k\ll \sqrt{|F_n|}$, writing $\mu$ for the measure induced by the filtration, we have 
\[
\E[E_n(\A,\A)] = \bigl(1+\cp_\mu(G)+\sq_\mu(G)\bigr)k^2 + o(k^2).
\]
In particular, we recover \cref{eq:intro-tao}, because $\cp_\mu(\mathbb Z)=1$ and $\sq_\mu(\mathbb Z)=0$, irrespective of the choice of measure. The same formula also gives the asymptotic bounds
\[
k^{2}\ \le\ \liminf_n \E[E_n(\A,\A)] \ \le\ \limsup_n \E[E_n(\A,\A)]\ \le\ 3k^{2}-2k.
\]
In the remainder of \cref{sec: 4}, we ask when these bounds are attained. For that purpose we treat two regimes: amenable groups, and finitely generated groups.

\Cref{subsec: 4-infinite-amenable} concerns amenable groups. Here, we draw on the existing literature about commuting probability in amenable groups (see, for instance, \cite{Tointon_commuting}). Therefore, our main task is to deal with the square-incidence probability. For a residually finite amenable group $G$, we show that $\sq_\mu(G)$ can be positive only if $G$ is virtually abelian, and that $\sq_\mu(G)$ is not dependent on $\mu$. The resulting statements, and the corresponding consequences for the limiting behaviour of the energy, are gathered in \cref{cor:multiplicativeamenable}.

In \cref{subsec: 4-infinite-fingen}, we assume that $G$ is finitely generated with a generating set $S$. The natural exhaustion is by word-metric balls $\B_S(n)$ around the identity.  In this setting, it remains an open question whether the commuting probability, as introduced in \cite{AMV17}, depends on the choice of $S$. We expect that the square incidence probability should vanish unless $G$ is virtually abelian. Since groups of subexponential growth are amenable with the balls forming a F\o lner sequence, we are left with considering the exponential-growth case. 
In \cref{conj:EarthboundZero}, we hypothesize that, for groups of exponential growth,
\[
\E[E_n(\A,\A)] = k^2 + o(k^2).
\]
We verify the veracity of \cref{conj:EarthboundZero} for several important families arising from geometric group theory: word-hyperbolic groups (\cref{ex: AA hyperbolic}), right-angled Artin groups (\cref{ex: AA RAAG}), and generalised lamplighter groups of the form $H \wr \Z$ (\cref{ex: AA lamp}).

In \cref{subsec:inf2}, we consider the pair $(\A,\A^{-1})$. The finite case suggests that the invariant controlling the asymptotic of $\E[E(\A,\A^{-1})]$ is the involution density $\iota_\mu(G)$, namely the probability that a randomly chosen element has order $2$, see \cref{eq: involution density}. Unfortunately, we are not able to prove it in full generality. Still, in \cref{question:golden}, we ask whether, in the regime $k\ll \sqrt{|F_n|}$,
\[
\E[E_n(\A,\A^{-1})] = \left(2+\iota_\mu(G)\right)k^2 + o(k^2)\,.
\]

Once again, we split the discussion about infinite groups into two. We begin with the amenable case (see \cref{subsec: 5-infinite-amenable}). We prove the previously conjectured asymptotic for residually finite amenable groups, and we also establish the basic properties of the involution probability. Namely, $\iota_\mu(G)$ is independent of the F\o lner sequence (see \cref{lemma: involution invariant}), and $\iota_\mu(G)\neq 0$ only if $G$ is virtually abelian (see \cref{lemma: involution charac}). These results, and the corresponding consequences for the expected energy, are summarised in \cref{cor:multiplicativeInverseBalls}.

Meanwhile, for the word-metric in finitely generated groups and the filtration by balls, the relevant case is again that of groups with exponential growth (see \cref{sec: 5-inffingen}). In this regime, \cref{conj:Jeff} asks whether
\[
\E[E_n(\A,\A^{-1})] = 2k^2 + o(k^2) \,.
\]
We prove this in several natural families from geometric group theory, notably word-hyperbolic groups (see \cref{ex: AAinv hyperbolic}) and certain generalised lamplighter groups (see \cref{ex: AAinv lamp}).

\subsection{Applications}

The energy estimates developed above can be combined with the inequality relating
energy and growth to obtain probabilistic results on the existence of subsets with
large growth.

As a first application, in \cref{thm: small squaring property},  we give an alternative proof (under the more restrictive hypotheses that the group $G$ is residually finite and amenable) of a theorem of Neumann on the small squaring property. More precisely, we prove that if for every finite subset $A \subseteq G$ we have $|A^{\ast 2}| <|A|^2$, then $G$ is finite-by-abelian-by-finite. 

We prove an approximate version of the existence of an \emph{additive basis} for a
finite group $G$ of minimal possible cardinality. (The definition of additive bases
and the known results on their cardinality are given in \cref{sec: 5-inffingen}.)
Explicitly, \cref{thm:largeDoublingFinite} states that, for every unbounded function
$h\colon\mathbb{N}\to\mathbb{N}$, and for every constant
$\epsilon>0$, provided that the finite group $G$ is large enough, there exists a subset $A$
such that
\[
|A| \le |G|^{\frac{1}{2}}h(|G|), \quad \hbox{and} \quad |A^{\ast 2}| \ge (1-\epsilon)|G| \,.
\]
To the best of our knowledge, this is the strongest result of this type that does not
rely on the Classification of Finite Simple Groups. By contrast, if one allows the use of the
Classification, $h$ can be taken to be a constant function and $\epsilon=0$.

Still in a Classification-free context, if the underlying group has no small
nontrivial linear representations (that is, if it is \emph{quasirandom}),
\cref{thm:largeDoublingFinite} can be upgraded to a diameter-type result: specifically, $G=A^{\ast 6}$ (see
\cref{cor:GowersTrick}).

Along the same line, \emph{thin bases} provide examples of subsets of the natural numbers with density zero, whose sumset has nevertheless positive density. (Definitions are given in \cref{subsec: infite growth}). In Theorem~\ref{thm:largeDoubling}, we prove an analogous phenomenon for direct limits of finite groups, respectively. More precisely, there exists a set $A$ such that, for a suitable left-invariant measure, $\mu(A)=0$ but $\mu(A^{\ast 2}) =1$. 

On another note, we also tackle a question inspired by the additive-combinatorial problem of determining whether
sets with $|AA^{-1}|\ |A^{\ast 2}|$ or sets with $|A^{\ast 2}|\ge |AA^{-1}|$ are more
common. (A brief account of the known results is given in
\cref{sec:sumDiffDom}.) In the regime $1 \ll k \ll \sqrt{|F_n|}$,
\cref{thm:sumDiffDom} shows that, with positive probability, both ratios
$|AA^{-1}|/|A^{\ast 2}|$ and $|A^{\ast 2}|/|AA^{-1}|$ are bounded by an absolute
constant. Our sampling model is quite different from that used in previous work on
sum- and difference-dominance. Most of those works rely on Bernoulli sampling models,
where each element is chosen independently. As a consequence, random sets typically
have cardinality comparable to $|F_n|$, and dominance phenomena arise from delicate
fringe effects. In contrast, our uniformly sampled $k$-subsets satisfy
$k \ll \sqrt{|F_n|}$, placing our results in a genuinely sparse regime. Moreover, our
result shows that the two cardinalities cannot differ too much too often, whereas
standard results typically compare these quantities only to the threshold value~$1$.

We believe that the restriction to first-moment estimates in our proof prevents us from addressing
finer questions of optimality, motivating the study of higher moments formulated in
\cref{question:higherMoments}.

\subsection{Acknowledgements}

We are grateful to Primo\v{z} Poto\v{c}nik for his help in developing the original proof
of \cref{cor:multiplicativeInverseFinite}, which ultimately led to the present work.
We also thank Urban Jezernik for many valuable discussions and for his helpful comments
on early drafts of this paper.

\section{Setup}

\subsection{Notation}
We collect here most of the standard notation we need for this paper.

\subsubsection*{Sets and asymptotics} For every set $X$, we denote its power set by $\mathcal{P}(X)$, and, for every positive integer $k$, we denote by $\mathcal{P}_k(X)$ the family of subsets of $X$ of cardinality $k$.
We write $x^{\underline \ell}$ for the \emph{Pochhammer symbol}, that is,
\[ x^{\underline \ell} =x(x-1)\dots(x-\ell+1) .\]
For integer-valued maps $f \colon \mathbb{R} \to \mathbb{Z}_{>0}$, we use the shorthand
$ \binom{N}{f} $ to mean $\binom{N}{f(n)}$ when the evaluation point $n$ is clear from the context.

We will repeatedly compare the limiting behaviour of real-valued functions defined on positive integers that are positive and eventually nondecreasing. For this purpose we use the standard Bachmann–Landau notation: given two such functions $f,g$, we write
\[\begin{split}
    &f \in o(g), \quad
    \hbox{if } \lim_{k} \frac{f(k)}{g(k)} = 0,\\
    & f \in O(g), \quad
    \hbox{if } \lim_{k} \frac{f(k)}{g(k)} \in (0 , +\infty],\\
    &f \in \Theta(g), \quad \hbox{if } \lim_{k} \frac{f(k)}{g(k)} \in (0, +\infty)\,.
\end{split}\]
Moreover, we will as well use Vinogradov's notation $f \ll g$ meaning $f \in O(g)$.

\subsubsection*{Group Theory} For every group $G$, and for every element $x \in G$, we write
\[
{\bf C}_G(x) = \{ g \in G \mid gx = xg \}
\]
for the centraliser of $x$ in $G$. If $G$ is a finitely generated group with a finite generating set $S$, we denote by ${\bf B}_S(n)$ the ball of radius $n$ around the identity element in the corresponding Cayley graph with respect to the associated word metric.

Given two subsets $A$ and $B$ of a group $G$, we write
\[ AB = \left\{ ab \mid a\in A, \, b\in B \right\}.\]
Moreover, if $A=B$, rather than writing $AA$, we prefer the symbol $A^{\ast 2}$. In this fashion, for $k$-fold products of this type of $A$ with itself, we write $A^{\ast k}$. Finally, we write 
\[\varinjlim G_n \] 
for direct limits, respectively. In our setting, the groups $G_n$ are finite. 

\subsubsection*{Probability} In probabilistic statements, we write $\mathbb{P}[A]$ for the probability of the event $A$, and $\mathbb{E}[{\sf X}]$ for the expectation of the random variable ${\sf X}$.

\subsection{Definition of energies} \label{sec:energy}
We start by defining the action energy, following Murphy's unpublished manuscript \cite{Murphy}. (The ideas from the manuscript have also been used in \cite{frenchDudes}.)
\begin{definition}
\label{def:actionenergy}
    Let $G$ be a discrete group, let $A\subseteq G$, let $\Omega$ be a discrete $G$-set, and let $\Delta \subseteq \Omega$. We define the \emph{action energy} of the pair $(A,\Delta)$ by
    \[ E(A,\Delta) = \#  \left\{ (a,b,\gamma,\delta) \in A^2 \times \Delta^2 \mid \gamma \cdot a = \delta \cdot b \right\}.\]
\end{definition}
If $\Omega = G$, and the action corresponds to the right regular permutational representation of $G$, then the action energy of $(A,\Delta)$ coincides with the \emph{multiplicative energy} of $(A,\Delta)$, as introduced by Tao in \cite{Tao2008}. Furthermore, if $G$ is abelian, we recover the usual \emph{additive energy} of $(A,\Delta)$ (see, for instance, \cite{TaoVu2006}).

Note that, if $A$ or $\Delta$ are infinite sets, the energy of the pair is automatically infinite. Thus action energy, as per its \cref{def:actionenergy} formulation, only provides information for finite sets.

Suppose that $A$ and $\Delta$ are finite. As
\begin{equation*} 
    \left\{ (a,a,\gamma,\gamma) \in A^2 \times \Delta^2 \right\} \subseteq \left\{ (a,b,\gamma,\delta) \in A^2 \times \Delta^2 \mid \gamma \cdot a = \delta \cdot b \right\} \,,
\end{equation*}
it follows that $E(A,\Delta) \ge |A||\Delta|$. Moreover, the equality in the inclusion above is attained if, and only if, every element of $\Delta \cdot A$ can be written in a unique way as a product $\gamma \cdot a$, for some $\gamma\in \Delta$ and $a\in A$. We can give two explicit examples of this behaviour.
\begin{example}\label{example:free}
    Let $F_r$ be the free group of rank $r$, and let $S$ be the set of the $r$ canonical generators of $F_r$. We interpret the right regular action of $F_r$ on itself as its action on the Cayley digraph $\mathrm{Cay}(F_r,S)$, and we choose $\Delta$ to be a $h$-subset of vertices whose pairwise distance is at least $2$. Then
    \[ E(S,\Delta) = \# \left \{(x,y,v,w) \in S^2 \times \Delta^2 \mid xv = yw \right\} = kh \,,\]
    because the equality of words $xv = yw$ holds if, and only if, $x = y$ and $v = w$. 
\end{example}
\begin{example}\label{example:Sidon}
A subset $A$ of an abelian group is called a \emph{Sidon set} if every element of the sumset $A+A$ has a unique representation as $a+c$ with $a,c \in A$. In this setting, the equation $a+c = b+d$ holds if, and only if, $a = b$ and $c = d$. Therefore,
    \[
        E(A,A)
        = \#\left\{(a,b,c,d) \in A^4 \mid a+c = b+d \right\} = k^2 \,.
    \]
\end{example}

Conversely, observe that, once the triplet $({a,b},\gamma) \in A^2 \times \Delta$ is fixed, then
\[\delta = \gamma \cdot ab^{-1} \in \Omega\]
is uniquely determined. As a consequence, we obtain the upper bound
\[ E(A,\Delta) \le |A|^2 |\Delta| \,.\]

\begin{remark}
    Note that the bound \[ E(A,\Delta) \le |A| |\Delta|^2\]
    does not hold in general, but only if the action of $G$ on $\Omega$ is \emph{semiregular}, that is, the stabilizer of every point is trivial. (It is also common in this context to say that the action is \emph{free}.) In particular, it holds for multiplicative and additive energies. Indeed, if we fix the triplet $(b,\gamma,\delta) \in A \times \Delta^2$ and we choose $g\in G$ such that $ \gamma \cdot  g = \delta$, then
    \[ \gamma \cdot  ab^{-1} = \delta, \quad \hbox{thus} \quad a \in G_\gamma gb \,.\]
    Therefore, 
    \[ E(A,\Delta) \le |A| |\Delta|^2 \max_{\gamma \in \Delta} |G_\gamma| \,.\]
\end{remark}

This consideration yields a natural normalisation constant $|A|^2|\Delta|$ for the notion of action energy, thus applicable to infinite sets whenever $G$ and $\Omega$ carry finite filtrations.

We focus on \emph{amenable groups} (for background and further references, we refer to \cite[Chapter~4]{CS}), that is, those groups admitting a F{\o}lner sequence $(F_n)$ -- namely a filtration series of finite non-empty subsets such that for all $g \in G$, 
\[
\lim_n \frac{|gF_n \triangle F_n|}{|F_n|} = 0.
\]

To such a F{\o}lner sequence, we can associate a finitely additive probability measure $\mu\colon \mathcal{P}(G)\rightarrow [0,1]$, defined by
\begin{equation}
\label{eq: define a measure with ultrafilter}
\mu(S)=\lim_{\omega}\frac{|S\cap F_n|}{|F_n|},
\end{equation}
where $\omega$ is a fixed non-principal ultrafilter on $\mathbb N$. The F{\o}lner property implies that $\mu$ is left-invariant, {\it i.e.}, 
$ \mu(gS)=\mu(S)$, for every $g\in G$ and every $S\subseteq G$. (We refer to Subsection~\ref{subsec:infinite} for further discussion.)

\begin{definition}
    Let $G$ be a countable amenable group, let $A\subseteq G$, let $\Omega$ be a countable $G$-set, and let $\Delta \subseteq \Omega$. For two given F\o lner sequences $(F_n)\subseteq G$ and $(\Phi_n) \subseteq \Omega$, and for a non-principal ultrafilter $\omega$ on $\N$, we define the \emph{normalized action energy} of the pair $(A,\Delta)$ by
    \[ \eta (A,\Delta) = \lim_\omega \frac{\#  \left\{ (a,b,\gamma,\delta) \in (A \cap F_n)^2 \times (\Delta \cap \Phi_n)^2 \mid \gamma \cdot a = \delta  \cdot b \right\}}{|F_n|^2|\Phi_n|} \in [0,1] \,.\]
\end{definition}
If $G$ or $\Omega$ are infinite, while $A$ or $\Delta$ are finite, the normalized action energy is always $0$. In contrast, for finite groups acting on finite domains, we obtain a renormalisation of the original energy, \emph{i.e.},
\[ \eta (A,\Delta)  = \frac{E(A,\Delta)}{|G|^2 |\Omega|}  \,.\]

\begin{remark}
We used an ultrafilter solely to ensure existence of the limiting values. Nevertheless, in our computations the expected energy is typically independent of this choice and, in fact, converges in the ordinary sense. We will write $\lim_n$ for proper limits, with convergence tacitly assumed whenever the notation is used.
\end{remark}

    Repeating the same reasoning we used for action energy, we can prove that the normalized action energy can be bounded by
    \[ \eta (A,\Delta) \le \mu(A)^2\mu(\Delta) \,.\]
    Note that, as a consequence, the normalized action energy is meaningful only when $A$ and $\Delta$ have both positive measure.

\subsection{Probability spaces and sampling}\label{sec:probability}
All our groups and all our domains will be discrete, and they will be endowed with two increasing \emph{filtration series} or \emph{exhaustions}, $(F_n)_{n \in \N} \subseteq G$ and $(\Phi_n)_{n \in \N} \subseteq \Omega$, that is,
\[ \bigcup_{n \in \N} F_n = G, \quad \bigcup_{n \in \N} \Phi_n = \Omega \,.\]
Furthermore, we will say that the filtration series are \emph{increasing} if
\[|F_n| < |F_{n+1}| , \quad |\Phi_n| < |\Phi_{n+1}| \,,\]
and that they are \emph{nested} if 
\[F_n \subseteq F_{n+1} , \quad \Phi_n \subseteq \Phi_{n+1} \,.\]

We define a \emph{filtered} version of energy by
\[ E_n(A,\Delta) = E(A\cap F_n, \Delta \cap \Phi_n) \,.\]
Observe that
\[ \lim_\omega E_n(A,\Delta) = E(A,\Delta) \,,\]
and, if $G$ is amenable and the two filtration series are F\o lner sequences, then 
\[ \lim_\omega \eta_n(A,\Delta) = \lim_\omega \frac{E_n(A,\Delta)}{|F_n|^2|\Phi_n|} = \eta(A,\Delta) \,.\]

We now need to set up the probability spaces we are considering in this paper. Let $X$ be a discrete set, let $(X_n)_{n \in \N} \subseteq X$ be a filtration of $X$, and let $k: \NN \rightarrow \NN$ be a nondecreasing function. We stress here that, to keep the notation light, we will write $k$ rather than $k(n)$ when the integer $n$ we are evaluating the function in is well understood from context. A \emph{$k$-subset of $X_n$ sampled uniformly at random}, $\A$, coincides with the random variable taking values in the set of $k$-subsets of $X_n$ with uniform probability.

Discussing about the convergence (in some sense) of the sequences $(\A)$ to a given random variable on $\Pa(X)$ is out of the scope of the current investigation. In most cases, we are still able to compute asymptotic expansions, as $n$ tends to infinity, for the quantities $\E[E_n(\A,\A)]$ under investigation.

\section{Action energy}
We are ready to give an upper and lower bound for the expected value of action energy.
\begin{thmx}\label{thm:action}
    Let $G$ be a discrete group, let $\Omega$ be a discrete $G$-set, let $k,h: \mathbb{N} \to \mathbb{N}$ be two nondecreasing functions, let $\A$ be a $k$-subset of $F_n$ sampled uniformly at random, and let $\D$ be a $h$-subset of $\Phi_n$ sampled uniformly at random. Then
    \[ kh \le \E[E_n(\A,\D)] \le kh\left( 1 + \frac{(k -1)(h-1)}{2(|\Phi_n| -1)} \right) \,.\]
    In particular, for every $\A$, $k$-subset of $G$ sampled uniformly at random, and for every $\D$, $h$-subset of $\Omega$ sampled uniformly at random, we have the following consequences.
    \begin{enumerate}[$(a)$]
        \item If $k$ and $h$ are bounded (and, with an abuse on notation, we denote their limits by $k$ and $h$, respectively), and $\Omega$ is finite,
        \[kh \le \lim_{n} \E[E_n(\A,\D)] \le k h\left( 1 + \frac{(k-1)(h-1)}{2(|\Omega| -1)} \right) \,.\]
        \item If $k$ and $h$ are bounded, and $\Omega$ is infinite,
        \[\lim_{n} \E[E_n(\A,\D)] = k h \,.\]
        \item If $k, h = \Theta (|F_n|)$, and $G$ is amenable,  
        then
        \[0 < \lim_\omega \E[\eta_n(\A,\D)] \le \frac{1}{2}  \,,\]
        and, otherwise,
        \[\lim_\omega \E[\eta_n(\A,\D)] = 0  \,.\]
    \end{enumerate}
\end{thmx}
\begin{proof}
    Consider the set
    \[ 
    \mathcal{S} = \left\{ \left(a,b,\gamma,\delta,A,\Delta\right) \in F_n^{\, 2} \times \Phi_n^{\, 2} \times \Pa_k(F_n) \times \Pa_h(\Phi_n) \mid (a,b) \in A^2, \, (\gamma,\delta) \in \Delta^2, \, \gamma \cdot a=\delta \cdot b \right\} \,.\]
    On the one hand,
    \begin{proofequation}\label{eq:proofAction}
        \E[E_n(\A,\D)] =  \binom{|F_n|}{k}^{-1}  \binom{|\Phi_n|}{h}^{-1}  |\mathcal{S}| \,.
    \end{proofequation}
    On the other hand,
    \[ \begin{split}|\mathcal{S}| &= \sum_{a,b \in F_n, \gamma \in \Phi_n} \# \left\{ (A,\Delta) \in \Pa_k(F_n) \times \Pa_h(\Phi_n) \mid (a,b)\in A^2, (\gamma, \gamma \cdot ab^{-1})\in \Delta^2\right\} \,.\end{split} \]

    We aim to compute the right hand side of the previous equality. To do so, we split $\mathcal{S}$ in two disjoint sets, $\mathcal{S}_0$ and $\mathcal{S}_1$, the former containing those tuples with $a$ and $b$ being equal, and the latter being distinct. If $a=b$, then $\gamma \cdot ab^{-1} = \gamma$. In particular,
    \[\begin{split}
        |\mathcal{S}_0| &= \sum_{a \in F_n, \gamma \in \Phi_n} \# \left\{ (A,\Delta) \in \Pa_k(F_n) \times \Pa_h(\Phi_n) \mid a \in A, \gamma \in \Delta\right\}
        \\&= |F_n| |\Phi_n| {|F_n|-1 \choose k-1}  {|\Phi_n|-1 \choose h-1}
        \\&=  k \, h \, {|F_n| \choose k}  {|\Phi_n|\choose h} \,.
    \end{split}\]
    Suppose now that $a$ and $b$ are distinct elements. We are not guaranteed that $\gamma \cdot ab^{-1}$ is an element of $\Phi_n$. Hence, we obtain
    \[\begin{split}
        |\mathcal{S}_1| &= \sum_{a \ne b \in F_n, \gamma \in \Phi_n} \# \left\{ (A,\Delta) \in \Pa_k(F_n) \times \Pa_h(\Phi_n) \mid (a,b) \in A^2, (\gamma, \gamma \cdot ab^{-1}) \in \Delta^2\right\}
        \\&= \sum_{a \ne b \in F_n, \gamma \in \Phi_n} \1_{\Phi_n}(\gamma \cdot ab^{-1}) {|F_n|-2 \choose k-2} {|\Phi_n|-2 \choose h-2}\,. 
    \end{split}\]
    This last sum might be zero (see \cref{example:free}); hence, we can only state $|\mathcal{S}_1| \ge 0$. Notwithstanding, the upper bound is given by
    \[\begin{split}
        |\mathcal{S}_1| &= \sum_{a \ne b \in F_n, \gamma \in \Phi_n} \1_{\Phi_n}(\gamma \cdot ab^{-1}) {|F_n|-2 \choose k-2} {|\Phi_n|-2 \choose h-2}
        \\&\le {|F_n| \choose 2}{|F_n|-2 \choose k-2} |\Phi_n| {|\Phi_n|-2 \choose h-2}
        \\&= \frac{k(k-1)h(h-1)}{2 (|\Phi_n|-1)} {|F_n| \choose k}  {|\Phi_n|\choose h} \,.
    \end{split}
    \]
    Observe that the equality is achieved whenever $\Omega$ is finite, and $n$ is large enough so that $\Omega = \Phi_n$, substituting in \cref{eq:proofAction}, this proves item $(a)$.
    
    In general, substituting the lower and upper bounds in \cref{eq:proofAction}, we get
    \[kh \le \E[E_n(\A,\D)] \le kh\left( 1 + \frac{(k-1)(h-1)}{2(|\Phi_n| -1)} \right) \,.\]
    The remainder of the statement of \cref{thm:action} follows from here.
\end{proof}

Observe that \cref{thm:action}~$(b)$ suggests that, if $k$ and $h$ are bounded and $\Omega$ is infinite, the expected behaviour of a generic pair $(A,\Delta)$ is that, for every point in $\Delta \cdot A$, there exists a unique way to express it in the form $\gamma \cdot a$, in analogy with \cref{example:free}. This is plausible because, as the points become farther apart, coincidences between their images tend to disappear, and the action begins to resemble the action of a free group on the infinite regular tree.

For the next example, we need to recall the following lemma, which is a generalisation to action energy of \cite[Equation~(9)]{Tao2008}.
\begin{lemma}\label{lemma:CauchySchwartz}
    Let $A$ be a subset of a discrete group $G$, and let $\Delta$ be a subset of a $G$-set $\Omega$. Then,
    \[ E(A,\Delta) \ge \frac{|A|^2|\Delta|^2}{|\Delta \cdot A|} \,.\]
\end{lemma}
\begin{proof}
    We define, for every $\omega\in \Omega$,
    \[ r(\omega) = \# \left\{ (a,\gamma) \in A \times \Delta \mid \omega = \gamma \cdot a \right\} \,.\]
    Observe that
    \[ \sum_{\omega\in \Delta \cdot A} r(\omega) = |A||\Delta|, \quad \hbox{and} \quad \sum_{\omega\in \Delta \cdot A} r(\omega)^2 = E(A,\Delta) \,.\]
    By Cauchy-Schwartz inequality,
    \[|A|^2|\Delta|^2 = \left( \sum_{\omega\in \Delta \cdot A} r(\omega) \right)^2 \le \left( \sum_{\omega\in \Delta \cdot A} 1 \right) \left( \sum_{\omega\in \Delta \cdot A} r(\omega)^2 \right) = |\Delta \cdot A| E(A,\Delta) \,.\]
    The result follows by dividing both sides by $|\Delta \cdot A|$.
\end{proof}

\begin{example}\label{example:approximate}
    Approximate groups are examples of sets with large multiplicative energies. (See \cite{Tointon} for an introduction to the topic.) Let $G$ be a discrete group, and let $A$ be a $\kappa$-approximate subgroup. Then, by \cref{lemma:CauchySchwartz},
    \[ \frac{1}{\kappa} \le \frac{|A|}{|A^{\ast 2}|} \le \frac{E(A,A)}{|A|^3} = \eta (A,A) \,. \]
    In particular, if $A$ is a subgroup, then $\eta(A,A) = 1$.
\end{example}

In view of \cref{example:approximate}, we can interpret \cref{thm:action}~$(c)$ as saying that a randomly sampled $k$-subset usually produces something whose normalised energy is lower than that of a $2$-approximate group. But there might be a great proportion of $2$-approximate groups and genuine subgroups. For instance, if $G$ is a locally finite group endowed with a F\o lner sequence $(F_n)$ consisting of finite groups, and $k \geq |F_n|/2$, then, for every $k$-subset $A_k \subseteq F_n$,
\[ |A_k^{\ast 2}| \le |F_n| \le 2k \,.\]
(An example from this class of groups is given by the finitary symmetric group ${\rm FSym}(\mathbb{Z})$, that is, the permutation group over the integers consisting only of permutations with finite support.)
On the other hand, as intuition suggests, this behaviour is quite rare. For instance, for a fixed $k$, the expected size of a $k$-subset of $C_2^n$ tends to $2^k$ as $n$ tends to infinity. As \cref{thm:action}~$(c)$ must account for the previously described phenomenon in its current generality, it cannot be improved.

\section{Multiplicative energy of a set with itself} \label{sec: 4}
We now restrict our attention to the regular action of a group by right multiplication, and, furthermore, we suppose that $A=\Delta$. On one hand, this mimics \cref{eq:intro-tao}. On the other hand, without this extra assumption, we are not currently able to sharpen \cref{thm:action} further. We obtain the following upper and lower bounds for the expected value of the multiplicative energy in this context.

\begin{thmx}\label{thm:multiplicative}
    Let $G$ be a discrete group, let $(F_n)_{n \in \N} \subseteq G$ be a symmetric finite filtration series of $G$, let $k \colon \mathbb{N} \to \mathbb{N}$ be a nondecreasing function, and let $\A$ be a $k$-subset of $F_n$ sampled uniformly at random. Then
    \[\begin{split}
        \E[E_n(\A,\A)] &\le \frac{k^{\underline 4}}{|F_n|-3} + \frac{k^{\underline 3}}{(|F_n|-1)^{\underline 2}}\left(4|F_n| -2 \right) +
        \\&\, +k^{\underline 2} \left( 2 + \frac{\max_{x\in F_n \setminus\{1\}} |{\bf C}_G(x) \cap F_n| -1}{|F_n|-1} \right) +k \,,
    \end{split}\]
    and
    \[\begin{split}
        \E[E_n(\A,\A)] &\ge \frac{k^{\underline 4}}{(|F_n|-1)^{\underline 3}} \left(|F_n|^2 - 5|F_n| +2\right) +
        \\&\, + \frac{k^{\underline 3}}{(|F_n|-1)^{\underline 2}} \left( |F_n| - \max_{x\in F_n \setminus\{1\}} |{\bf C}_G(x) \cap F_n|\right) + k^2\,.
    \end{split}\]
\end{thmx}

Before the proof, we extract a result that will be of use later on.

\begin{lemma}\label{lemma:andoniRevenge}
Let $(F_n)_{n \in \N}$ be a symmetric filtration series in a group $G$. Then, 
\[ \sum_{ y \in F_n^{*2}} |F_n \cap y F_n | = |F_n|^2 \,.\]
\end{lemma}
\begin{proof}
Note that 
\[ F_n \cap y F_n = \{  x  \in F_n\mid x \in y F_n\} = \{ (x,z) \in F_n^2 \mid xz^{-1}= y \}.\]
We define, by also using the hypothesis that $F_n$ is symmetric,  the map 
\[ \Psi \colon F_n^2 \rightarrow F_n^{*2}, \quad (x,z) \mapsto xz^{-1} \,.\]
The fibers of $\Psi$,
\[ \Psi^{-1}(y) = \left\{ (x,z) \in F_n^2 \mid xz^{-1} = y \right\} \,,\]
partition $F_n^2$. Therefore,
\[ \sum_{ y \in F_n^{*2}} |F_n \cap y F_n | = \sum_{ y \in F_n^{*2}} \left|\Psi^{-1}(y) \right| = |F_n|^2 \,. \qedhere \]
\end{proof}

\begin{proof}[Proof of \cref{thm:multiplicative}]
We use the same double counting argument as in the proof of \cref{thm:action}. We obtain
    \begin{proofequation}\label{eq:regular2}
        \E\left[ E(\A, \A) \right] = {|F_n| \choose k}^{-1} \sum_{(a,b,c)\in F_n^3} \# \left\{ A \in \Pa_k(F_n)  \middle|  (a \,, \, b, \, c, \, c^{-1}ab) \in A^4 \right\} \,.
    \end{proofequation}
    Our approach is, thus, to determine the sum on the right-hand side of Equality~\eqref{eq:regular2}. To do so, we partition the set of triplets of $F_n^3$ into nine disjoint sets, with the property that the number of $k$-subsets $A$ containing a triplet from a specified set is easy to compute. Explicitly, we define
    \begin{align*}
        Q_1^{(1)} &= \{(a,b,c)\in F_n^3\mid \#\{a,b,c\}  =3,\, ab = ca \}\,,\\
        Q_1^{(2)} &= \{(a,b,c)\in F_n^3\mid \#\{a,b,c\}  =3,\, ab = c^2 \}\,,\\
        Q_1^{(3)} &= \{(a,b,c)\in F_n^3\mid \# \{a,b,c\}  =3,\, ab \neq ca,\, ab \neq c^2\}\,,\\
        Q_2 &= \{(a,b,a)\in F_n^3\mid a\neq b \}\,,\\
        Q_3^{(1)} &= \{(a,b,b)\in F_n^3\mid a\neq b,\, ab = ba \}\,,\\
        Q_3^{(2)} &= \{(a,b,b)\in F_n^3\mid a\neq b,\, ab \neq ba \}\,,\\
        Q_4^{(1)} &= \{(a,a,c)\in F_n^3\mid a\neq c, \, a^2=c^2 \}\,,\\
        Q_4^{(2)} &= \{(a,a,c)\in F_n^3\mid a^2\neq c^2 \}\,,\\
        Q_5 &= \{(a,a,a) \in F_n^3\}\,.
    \end{align*}
    Note that the element $c^{-1}ab$ is distinct from those appearing in the triplet for $Q_1^{(3)}$, $Q_3^{(2)}$ and $Q_4^{(2)}$, while it coincides with an element of the triplet otherwise. We now need to establish the cardinalities of each set of the partition. Preliminarily, we observe that
    \[ |Q_2| = |F_n|( |F_n| -1) , \quad \hbox{and} \quad |Q_5| = |F_n| \,.\]

    We now focus on giving an upper and lower bound for the triplet $|Q_1^{(1)}|$, $|Q_1^{(2)} |$, and $|Q_1^{(3)} |$. We start by deriving an upper bound for the cardinality of $Q_1^{(1)}$.
    As $F_n$ is symmetric, for every $y \in F_n$, we define an injective map
    \[\psi_y :\,  F_n  \to F_n \times F_n y \times y F_n , \, \quad x \mapsto (x, x^{-1}y, yx^{-1} ) \,.\]
    We claim that
    \begin{proofequation} \label{eq: Q11} \left( \bigcup_{y \in F_n^{*2} } \left\lbrace \psi_y(x) \mid x \in F_n \setminus {\bf C}_G(y) \right\rbrace  \right) \cap F_n^3  = Q_1^{(1)} \,.\end{proofequation}
    On the one hand, if $y \in F_n^{\ast2}$ and $(x,x^{-1} y, yx^{-1} ) \in F_n^3$, then
    \[ x x^{-1} y = y = y x^{-1} x \,.\]
    The condition $x \notin {\bf C}_G(y)$ implies that $x$, $x^{-1}y$ and $yx^{-1}$ are pairwise distinct, and hence $(x,x^{-1}y, yx^{-1}) \in Q_1^{(1)}$.
    On the other hand, for every $(a,b,c) \in Q_1^{(1)}$, let $x = a \in F_n$ and $y = ab \in F_n^{*2}$. Then, $b = x^{-1}y $ and
    \[ c = aba^{-1} = x x^{-1}yx^{-1} = y x^{-1}\,. \]
    If $(ab)^a=ab$, then $c = b^a = b$, against the choice of $(a,b,c)$. Therefore, $(a,b,c)= \psi_y(x)$ for $y \in F_n^{\ast2}$ and $x \in F_n \setminus {\bf C}_G(y)$. This proves the claim.

    Observe that  for $x \in F_n\setminus {\bf C}_G(y)$, $\psi_y(x) \in F_n^3$ if, and only if, $x \in (F_n \cap F_n y \cap y F_n) \setminus {\bf C}_G(y)$, implying that
    \[ |\psi_y(F_n \setminus {\bf C}_G(y)) \cap F_n^3| = \left| (F_n \cap y F_n \cap y^{-1} F_n ) \setminus {\bf C}_G(y) \right| \leq |F_n \cap y F_n| \,.\] 
    Therefore, by \cref{lemma:andoniRevenge},
    \[ |Q_1^{(1)}| \le \sum_{y\in F_n^{\ast 2}} |F_n \cap y F_n| = |F_n|^{2}  \,. \]
    Observe that, if $G$ is abelian, $Q_1^{(1)} $ is empty. In particular, no better lower bound that  $|Q_1^{(1)}|\ge 0$ can be derived.

    We now turn our attention to $|Q_1^{(2)}|$. Observe that, in the setting of \cref{example:free}, $|Q_1^{(2)}|=0$. Thus, we only focus on giving an upper bound. We define
    \[r_n(x) = \#\{(a,b) \in F_n^2 \mid ab = x\}, \quad \hbox{and} \quad S= \{c^2 \mid c \in F_n\} \,.\]
    Note that, by double counting,
    \[E(F_n,F_n) = \sum_{x \in G}r_n(x)^2 \,.\]
    Using Cauchy-Schwartz inequality and the general upper bound on energy, we compute
    \begin{align*}
        |Q_1^{(2)}| &\le \sum_{c \in F_n}r_n(c^2) \\
        &= \sum_{x \in G}r_n(x)\1_S(x) \\
        &\le \left( \sum_{x \in G}r_n(x)^2 \right)^{1/2}\left( \sum_{x \in G}\1_S(x)^2 \right)^{1/2} \\ 
        &= E(F_n,F_n)^{1/2}|S|^{1/2} \\
        &\le |F_n|^{3/2}|F_n|^{1/2} \\
        &= |F_n|^2\,.
    \end{align*}
    
    Observe that
    \[ |Q_1^{(1)}| + |Q_1^{(2)}| + |Q_1^{(3)}| = |F_n|(|F_n| -1)(|F_n|-2) \,.\]
    Hence,
    \[|F_n| \left( |F_n|^2 - 5 |F_n| + 2  \right) \le |Q_1^{(3)}|  \leq |F_n|(|F_n| -1)(|F_n|-2) \,. \]

    We turn our attention to the pair $Q_3^{(1)}$ and $Q_3^{(2)}$. The analysis is significantly more straightforward than in the other case. We note that
    \[ |Q_3^{(1)}| + |Q_3^{(2)}| = |F_n|(|F_n| -1), \quad \hbox{and} \quad |Q_3^{(1)}| = \sum_{x \in F_n} (|{\bf C}_G(x) \cap F_n| - 1) \,.\]
    It follows that
    \[0 \le |Q_3^{(1)}| \le |F_n|  \left( \max_{ x \in F_n \setminus \{1\} }|{\bf C}_G(x) \cap F_n| -1 \right) \,,\]
    and thus
    \[|F_n|\left( |F_n| - \max_{ x \in F_n \setminus \{1\} }|{\bf C}_G(x) \cap F_n| \right) \le |Q_3^{(2)}| \le |F_n|(|F_n| -1) \,,\]

    Lastly, rather than an explicit computation of the cardinalities of $Q_4^{(1)}$ and $Q_4^{(2)}$, for the purpose of proving \cref{thm:multiplicative}, we are satisfied with the obvious bounds. Therefore, since
    \[ |Q_4^{(1)}| + |Q_4^{(2)}| = |F_n|(|F_n|-1) \,,\]
    we obtain that, for $i \in \{1,2\}$,
    \[ 0 \le |Q_4^{(i)}| \le |F_n|(|F_n|-1) \,.\]

    To compute the expected value of $E(\A,\A)$, we need to establish how many subsets $A$ contain a triplet $(a,b,c) \in Q_i$ depending on the condition that defines $Q_i$. Our initial consideration on the properties of the set of the partition shows that
    \begin{enumerate}[$(i)$]
        \item for $Q_1^{(3)}$, $\#\{a,b,c,d\}=4$;
        \item for $Q_1^{(1)}$, $Q_1^{(2)}$ , $Q_3^{(2)}$, and $Q_4^{(2)}$, $\#\{a,b,c,d\} = 3$;
        \item for $Q_2$, $Q_3^{(1)}$, and $Q_4^{(1)}$, $\#\{a,b,c,d\} = 2$;
        \item finally, for $Q_5$, $\{a,b,c,d\}=\{a\}$;
    \end{enumerate}

    Therefore, by substituting in \cref{eq:regular2}, we obtain
    \begin{align*} 
    \E\left[ E(\A, \A) \right] =& {|F_n| \choose k}^{-1}\left[ {|F_n|-4\choose k-4} |Q_1^{(3)}| + {|F_n|-3\choose k-3} (|Q_1^{(1)}|+ |Q_1^{(2)}| + |Q_3^{(2)}| +\right. \\
    &\left. +|Q_4^{(2)}|) + {|F_n|-2\choose k-2} (|Q_2| + |Q_3^{(1)}| + |Q_4^{(1)}|) + {|F_n|-1\choose k-1} |Q_5| \right] \,.\end{align*}
   Substituting the upper bounds for each subset we have computed, we obtain the final formula.
\end{proof}

We stress that, as a consequence of \cref{thm:multiplicative}, we have that, if $k \ll \sqrt{|F_n|}$, we have that
\[ k^2 \le \E[E_n(\A ,\A )] - O\left(\frac{k^4}{|F_n|}\right) \le 3k^2  \,,\]
and, if $k \ll \sqrt[3]{|F_n|}$,
\[ k^2 \le \E[E_n(\A,\A)] - O\left(\frac{k^4}{|F_n|}\right) \le 3k^2 -2k  \,.\]
These asymptotic bounds cannot be improved without adding extra assumptions to the model. This is showcased in the following examples, where we assume $k \ll \sqrt[3]{|F_n|}$ throughout.
\begin{example}\label{example:int}
    Let $G$ be the vector space over the field with two elements of infinite countable dimension, \emph{i.e.}, $G= C_2^{\aleph_0}$, and let $F_n$ be a filtration series in finite subsets of increasing cardinality. Using the notation of the proof of \cref{thm:multiplicative}, since $G$ is elementary abelian,
    \[ |Q_4^{(2)}| =0, \quad \hbox{ and } \quad |Q_4^{(1)}|  = |F_n|(|F_n|-1) \,,\]
    and, since every element has order $2$,
    \[ |Q_1^{(1)}| =  |Q_1^{(2)}| =0 \quad \hbox{ and } \quad  |Q_1^{(3)}|  = |F_n|(|F_n|-1)(|F_n|-2) \,.\]
   Thus, both the upper and lower bounds are asymptotically equivalent, to be precise,
    \[  \lim_n \frac{\E[E_n(\A, \A)]}{3k^2 - 2k} = 1  \,.\]
\end{example}
\begin{example}\label{example: heissenberg}
    The Heisenberg group 
    \[ H(\mathbb{Z}) =
    \left\{
\begin{pmatrix}
1 & a & c \\
0 & 1 & b \\
0 & 0 & 1
\end{pmatrix}
\,\middle\vert\,
\, a, \, b, \, c \in \mathbb{Z}
\right\} \,.
    \] 
    is an amenable group, and with respect to a generic F\o lner filtration, satisfies
    \[  \lim_n \frac{\E[E_n(\A, \A)]}{k^2} = 1 \,.\]
    This will be proved in the \cref{subsec:infinite}, using \Cref{cor:multiplicativeamenable}.
\end{example}

\subsection{Finite groups}\label{sec:char}
Quite charmingly, while writing the proof, it appears clear that the result can be sharpened by knowing how many square roots does each set from the filtration series contain. If the group is finite, this is, eventually, counting the number of square roots in the group. This question can be addressed using character theory (through the folkloristic \cref{lemma:squareRoots}, which essentially mimics \cite[Lemma~4.4]{Isaacs}). Recall that, for an irreducible character $\chi$ of $G$, the \emph{Schur-Frobenius index} of $\chi$ is defined by
\[ {\bf s}(\chi) = \frac{1}{|G|} \sum_{g\in G} \chi(g^2) \,.\]
The range of the operator $\mathbf{s}$ is quite limited, consisting of the set $\{-1,0,1\}$ (see \cite[Theorem~4.5]{Isaacs}). Moreover, we can partition the irreducible characters in three classes depending on the value of ${\bf s}(\chi)$. Namely,
if ${\bf s}(\chi)=1$ we say that $\chi$ is \emph{real}, if ${\bf s}(\chi)=0$ that $\chi$ is \emph{complex}, and if ${\bf s}(\chi)=-1$ that $\chi$ is \emph{quaternionic}. For a finite group $G$, we denote by $\kappa(G)$ the number of irreducible characters of $G$ (or, equivalently, the number of conjugacy classes of $G$, see \cite[Corollary~2.5]{Isaacs}), and we define $\varepsilon(G)$ to be the number of irreducible characters of real and quaternionic (but not complex) type. Clearly, $\varepsilon(G) \le \kappa(G)$, and as the Schur-Frobenius indexes are explicitly computable once the character table of $G$ is known, the same is true for $\varepsilon(G)$ (see the proof of \cref{lemma:characterMagic}).

The main result of \cref{sec:char} is the following. We will prove it after some auxiliary lemmas.
\begin{corx}\label{cor:multiplicativeFinite}
     Let $G$ be a finite group, let $k \in \mathbb{N}$ be a positive integer with $k \leq |G|$, and let $\A$ be a $k$-subset of $G$ sampled uniformly at random. Then,

    \[\begin{split}
        \E[E(\A,\A)] &= \frac{k^{\underline{4}}}{(|G|-1)^{\underline{3}}} 
    \left( |G|^2 - 5|G| + \varepsilon(G) + \kappa(G) + 3 \right)
\\&\, + \frac{2\,k^{\underline{3}}}{(|G|-1)^{\underline{2}}} 
    \left( 2|G| - \varepsilon(G) - \kappa(G) -1 \right)
\\&\, + \frac{k^{\underline{2}}}{|G|-1} 
    \left( \varepsilon(G) + \kappa(G)  -1  \right) 
+ k^2 \,.
    \end{split}\] 
\end{corx}

The following two lemmas are well-known: we prove them here, because we cannot locate easily accessible proofs in the literature.

\begin{lemma}\label{lemma:sumCentralizers}
    Let $G$ be a finite group. Then,
    \[\sum_{y\in G} |{\mathbf C}_G(y)| =  \kappa(G) |G| \,.\]
\end{lemma}
\begin{proof}
    Let $C \in \mathrm{Cl}(G)$ be a conjugacy class containing an element $y\in C$. By the Orbit Stabilizer Lemma, \[ |C| |{\mathbf C}_G(y)|=|G|\,.\] Hence, summing first over each conjugacy class, and then over the $\kappa(G)$ conjugacy classes,
    \[ \sum_{y\in G} |{\mathbf C}_G(y)| = \sum_{C \in \mathrm{Cl}(G)} |C| |{\mathbf C}_G(y)| = \sum_{C \in \mathrm{Cl}(G)}  |G| = \kappa(G) |G| \,. \qedhere\]
\end{proof}
\begin{remark}
\label{remark:cp}
Observe that \cref{lemma:sumCentralizers} yields that
    \[ \frac{\kappa(G)}{|G|} =
    \frac{|\{(x,y)\in G^2:\ xy=yx\}|}{|G|^2} = 
    \mathrm{cp}(G) \,,\]
where $\cp$ is the so-called \emph{commuting probability} of a finite group. This concept was introduced by Gustafson \cite{Gustafson73} in his study of the maximum number of commuting pairs in a nonabelian group, and it has since received considerable attention in the literature (see, for instance, \cite{fincom4,fincom3,fincom1,fincom5,fincom2}).
\end{remark}

For an element $g\in G$ of a finite group $G$, we denote the number of its square roots by $r(g)$, that is,
\[r(g) = \# \left\{ x\in G \mid x^2=g \right\} \,.\]

\begin{lemma}\label{lemma:squareRoots}
    Let $G$ be a finite group, and let $g\in G$ be an element. Then,
    \[r(g) = \sum_{\chi\in\mathrm{Irr}(G)} \mathbf{s}(\chi) \chi(g) \,.\]
\end{lemma}
\begin{proof}
    As the map $r: G \to \mathbb{N}$ is a class function,
    \[ r(g) = \sum_{\chi\in\mathrm{Irr}(G)} \langle r, \chi \rangle \chi(g) \,.\]
    Hence, we are left to determine the inner products $\langle r, \chi \rangle$.  Notwithstanding, note that
    \[ r(g) = \sum_{x\in G} \1_g(x^2)\,.\]
    It follows that
    \[ \langle r, \chi \rangle  = \frac{1}{|G|} \sum_{y \in G} r(y) \chi(y) = \frac{1}{|G|} \sum_{y \in G} \sum_{x\in G} \1_y(x^2) \chi(y) =  \frac{1}{|G|} \sum_{x\in G} \chi(x^2) = \mathbf{s}(\chi) \,.\]
    Substituting back in the first formula for $r(g)$, the proof is complete.
\end{proof}

We extract another character theoretic equality before diving into the proof of \cref{cor:multiplicativeFinite}.

\begin{lemma}\label{lemma:characterMagic}
    Let $G$ be a finite group. Then,
    \[\sum_{g\in G} r(g)^2 = \varepsilon(G) |G| \,.\]
\end{lemma}
\begin{proof}
    Using \cref{lemma:squareRoots}, we compute
    \[ \sum_{g\in G} r(g)^2 = \sum_{g\in G} \left(\sum_{\chi \in\mathrm{Irr}(G)} \mathbf{s}(\chi) \chi(g) \right)^2 = \sum_{\chi, \psi \in \mathrm{Irr}(G)} \left( \mathbf{s}(\chi)\mathbf{s}(\psi) \sum_{g\in G}  \chi(g)\psi(g) \right) \,.\]
    Using character orthogonality, and the fact that ${\bf s}(\chi)^2 = 0$ if $\chi$ affords a complex representation, and ${\bf s}(\chi)^2 = 1$ otherwise, we obtain 
    \[ \sum_{g\in G} r(g)^2 = |G| \sum_{\chi \in\mathrm{Irr}(G)} \mathbf{s}(\chi)^2 = \varepsilon(G) |G|  \,,\]
    as desired.
\end{proof}

\begin{remark}\label{remark:sq}
    In the same spirit as \cref{remark:cp}, we can interpret the ratio $\varepsilon(G)/|G|$ as the probability that two group elements share the same square. (We will call this quantity \emph{square incidence probability}.)
    Indeed, counting pairs $(x,y)\in G^{2}$ with $x^{2}=y^{2}$ by their common square $g$ yields
    \[ \left|\{ (x,y) \in G^2 \mid x^2 = y^2 \} \right| =   \sum_{g \in G} \left|\{x \mid x^2 = g \}\right|^2 = \sum_{g \in G} r(g)^2 \,,\]
    and, by \cref{lemma:characterMagic}, we conclude
    \[  \left|\{ (x,y) \in G^2 \mid x^2 = y^2 \} \right| = |G| \varepsilon(G) \,.\]
    Therefore,
    $$\frac{\varepsilon(G)}{|G|} = \frac{ \left|\{ (x,y) \in G^2 \mid x^2 = y^2 \} \right|} {|G|^2} \,.$$
\end{remark}

\begin{proof}[Proof of \cref{cor:multiplicativeFinite}]
    The finiteness of $G$ implies that we can assume that, for $n$ large enough, $F_n = G$. In this regime, we need to obtain explicit expressions for the cardinalities of all the sets of the partition.

    We start with the pair $Q_4^{(1)}$ and $Q_4^{(2)}$. An elementary counting argument (see \cref{remark:sq}) gives, together with \cref{lemma:characterMagic},
    \[ |Q_4^{(1)}| = \sum_{g\in G} r(g)^2 = \varepsilon(G) |G| \,.\]
    Since $|Q_4^{(1)}| + |Q_4^{(2)}| = |G|(|G|-1)$, then
    \[ |Q_4^{(2)}| = |G|(|G|- 1 - \varepsilon(G)) \,.\]
    
    From the computation of $|Q_4^{(1)}|$, we can promptly derive the number of triplets in $Q_1^{(2)}$. Indeed, observe that the number of triplets $(a,b,c)\in G^3$, with $a \ne c$, such that $ab=c^2$ is $|G|(|G|-1)$, as $b$ is uniquely determined by the choice of $a$ and $c$. Note that $b\ne c$, while the possibility $b=a$ is not excluded, but, in such case, $(a,b,c) \in Q_4^{(1)}$. Therefore,
    \[ |Q_1^{(2)}| = |G|(|G|-1) - |Q_4^{(1)}| = |G|(|G|-1-\varepsilon(G)) \,.\]

    We now deal with $|Q_1^{(1)}|$ using \cref{eq: Q11}.  Since, for every group element $y\in G$, $G=yG=Gy$, \cref{lemma:sumCentralizers} implies 
    \[\begin{split}
        |Q_1^{(1)}| &= \sum_{y \in G} |(G \cap Gy \cap yG) \setminus {\bf C}_G(y)|
        \\&= \sum_{y \in G} |G \setminus {\bf C}_G(y)|
        \\&= |G|^2 - \sum_{y \in G} |{\bf C}_G(y)|
        \\&= |G| (|G| - \kappa(G)) \,.
    \end{split}  \]
    It follows that
       \[ |Q_1^{(3)}| = |G|(|G|^2 - 5|G| + 3 + \varepsilon(G) + \kappa(G))\,.\]

    Finally, using again \cref{lemma:sumCentralizers},
    \[ |Q_3^{(1)}| = \sum_{y \in G} |{\bf C}_G(y)|- |G| = |G|(\kappa(G) - 1 ), \quad \hbox{and} \quad |Q_3^{(2)}| = |G|( |G| - \kappa(G)) \,.\]
As every set of the partition has an explicitly computed cardinality, repeating the computations from the proof of \cref{thm:multiplicative} gives the conclusion.
\end{proof}

As mentioned earlier, if the character table of a group is known, we can pivot on \cref{cor:multiplicativeFinite} to compute $\E[E(\A,\A)]$.
\begin{example}
    Let $G = \mathrm{GL}_2(q)$. For $q$ even, $\varepsilon(G) = q+1$ and $\kappa(G) = q^2 - 1$. (Observe that we can obtain $\varepsilon(G)$ by a direct computation using the formulae above.) It follows that
    \[\begin{split}
        \E[E(\A,\A)] &=\frac{\left(3 + q + q^2 - 5(q-1)^2 q(1+q) + (q-1)^4 q^2 (1+q)^2\right) k^{\underline{4}}}
     {\left((q-1)^2 q(1+q) - 1\right)^{\underline{3}}} \\
&+ \frac{2\left(-1 - q - q^2 + 2(q-1)^2 q(1+q)\right) k^{\underline{3}}}
       {\left((q-1)^2 q(1+q) - 1\right)^{\underline{2}}} \\
&+ \frac{(-1 + q + q^2)\, k^{\underline{2}}}{(q-1)^2 q(1+q) - 1}
+ k^2 \,.
\end{split}\]
    For $q$ odd, similarly, we obtain that $\varepsilon(G) = q+3$ and $\kappa(G) = q^2 - 1$. Hence, 
     \[\begin{split}
        \E[E(\A,\A)] &= \frac{\left(3 + q + q^2 - 5(q-1)^2 q(1+q) + (q-1)^4 q^2 (1+q)^2\right) k^{\underline{4}}}
     {\left((q-1)^2 q(1+q) - 1\right)^{\underline{3}}} \\
&\quad + \frac{2\left(-1 - q - q^2 + 2(q-1)^2 q(1+q)\right) k^{\underline{3}}}
       {\left((q-1)^2 q(1+q) - 1\right)^{\underline{2}}} \\
&\quad + \frac{(-1 + q + q^2)\, k^{\underline{2}}}{(q-1)^2 q(1+q) - 1}
+ k^2 \,.
\end{split}\]
\end{example}
We also have a strong control on the asymptotics of $\E[E(\A,\A)]$ for large finite groups. In fact, let $(G_n)$ be a sequence of finite groups (of increasing order), let $k$ be a positive integer, and let $(\A)$ be a sequence of $k$-subsets of $G_n$ sampled uniformly at random. The asymptotics along $(G_n)$ of the formula appearing in \cref{cor:multiplicativeFinite} can be fully understood from the asymptotics of $\varepsilon(G_n) + \kappa(G_n)$. Indeed,
\[ \E[E(\A,\A)] = \left( 1 + \frac{\varepsilon(G_n) + \kappa(G_n)}{|G_n|-1} \right)k^2 - \frac{\varepsilon(G_n) + \kappa(G_n)}{|G_n|-1} k + O\left( \frac{1}{|G_n|} \right)\]
For instance, if the sequence $(G_n)$ consists of nonabelian simple groups, then $\kappa(G_n)$ is sublinear in $|G_n|$. Thus, since $\varepsilon(G) \leq \kappa(G)$,
\begin{equation} \label{eq:simple}
    \E[E(\A,\A)] = k^2 + \Theta\left(\frac{\varepsilon(G_n) + \kappa(G_n)}{|G_n|}\right) \,.  
\end{equation}
Meanwhile, if the groups $G_n$ are abelian of odd order, $\kappa(G_n) = |G_n|$ and $\varepsilon(G_n)=0$, and we find
\[\E[E(\A,\A)] = 2k^2 - k  + \Theta\left(\frac{1}{|G_n|}\right) \,.\]

\begin{remark}\label{remark:previousParagraph}
    We can characterize the sequences of finite groups where $\lim_n \E[E(\A,\A)] = k^2$, by going back to the interpretation of the ratio $\kappa(G_n) / |G_n|$ as commuting probability, as explained in \cref{remark:cp}. Neumann proved in \cite[Theorem 1]{Neu89} that, for some positive constant $\alpha >0$, $\mathrm{cp}(G_n)>\alpha$ if, and only if, $G_n$ is ($\alpha$-bounded)-by-abelian-by-($\alpha$-bounded). (Here, \emph{$\alpha$-bounded} means finite of order at most $f(\alpha)$ for some function $f$ depending only on $\alpha$.)  Therefore, as $\varepsilon(G_n) \leq \kappa(G_n)$, the quantity $\kappa(G_n)+\varepsilon(G_n)$ grows linearly in $|G_n|$ if, and only if, each $G_n$ is finite-by-abelian-by-finite, where the finite sections are uniformly bounded.
\end{remark}

\subsection{Infinite groups} 
\label{subsec:infinite}
We have observed in \cref{remark:cp,remark:sq} that, for a finite group, the quantity $\E[E(\A,\A)]$ crucially depends on the ratios $\kappa(G)/|G|$ and $\varepsilon(G)/|G|$, which are the probabilities that a pair of elements $x,y\in G$ sampled uniformly at random solve the equations $xy=yx$ and $x^2=y^2$, respectively. This probabilistic interpretation extends \emph{verbatim} to infinite groups upon replacing the uniform counting measure by a finitely additive probability measure $\mu$ on $G$.  In this section, our main objective is to develop the corresponding notions for infinite groups.

Let $w\in F_k$ be a  word in $k$ variables. For a fixed group $G$, we denote by
\[
w \colon G^k\longrightarrow G
\]
the associated word map defined by substitution. The \emph{$w$-probability} of $G$ with respect to a finitely additive measure $\mu$ is
\[
\mathbb P_\mu(w;G)
=
\mu ^{\otimes k} \Bigl(\bigl\{(g_1,\dots,g_k)\in G^k \,\bigm|\, w(g_1,\dots,g_k)=1\bigr\}\Bigr).
\]
For instance, the \emph{commuting probability}, which is the $w$-probability of the commutator word $w=\gamma_2=[x,\, y ]$, or, explicitly,
\begin{equation} \label{eq: commuting proba}
\mathrm{cp}_\mu(G) =
(\mu\otimes\mu)\Bigl(\bigl\{(x,y)\in G^2 \,\bigm|\, xy=yx\bigr\}\Bigr)
 = 
\int_G \mu\bigl(\mathbf{C}_G(x)\bigr)\,d\mu(x) \,,
\end{equation}
has been extensively studied (see \cite{Tointon_commuting} for the amenable case, and \cite{AMV17} for the filtration arising from balls in the word metric of a finitely generated group).

Any exhaustion provides a finitely additive measure. Indeed, let $(F_n)_{n\in\mathbb N}$ be an increasing sequence of finite subsets whose union is $G$. Each $F_n$ defines the finitely supported probability measure
\[
\mu_n(X)=\frac{|X\cap F_n|}{|F_n|},
\quad X\subseteq G,
\]
or, equivalently, the averaging functional on $\ell^\infty(G)$
\[
f\mapsto \frac{1}{|F_n|}\sum_{g\in F_n} f(g).
\]
By the Banach--Alaoglu Theorem, the unit ball of $\ell^\infty(G)^*$ is weak-$*$ compact, and thus the above sequence admits a weak-$*$ convergent subsequence. Any such subsequential limit defines a mean $m$ on $\ell^\infty(G)$, namely
\[
m(f)=\lim_{k\to\infty}\frac{1}{|F_{n_k}|}\sum_{g\in F_{n_k}} f(g),
\quad f\in\ell^\infty(G) \,.
\]
Therefore, a corresponding finitely additive probability measure $\mu$ on $G$, given by
\[
\mu(X)=m(\mathbf 1_X)=\lim_{k\to\infty}\mu_{n_k}(X),
\quad X\subseteq G.
\]

\begin{lemma}\label{lemma:cP}
Let $G$ be a group, and let $(F_n)_{n \in \N}$ be a filtration series in $G$. Then, with respect to any non-principal ultrafilter $\omega$ in $\N$,
\[ \lim_{\omega} \frac{1}{|F_n|^2} \sum_{x \in F_n} |{\bf C}_G(x) \cap F_n| = \cp_\mu(G) \,.\]
\end{lemma}
\begin{proof}

We define the commuting-indicator function as
\[
   \mathbf{c}(x,y) = 
  \begin{cases}
    1 &\text{if } xy = yx \,,\\
    0 &\text{otherwise} \,.
  \end{cases}
\]
Then,
\[
  \frac{1}{|F_n|^2} \sum_{x \in F_n} |{\bf C}_G(x) \cap F_n| =
 \frac{1}{|F_n|^2} \sum_{x, \, y \in F_n}  
    \mathbf{c}(x,y) \,.
\]
By taking the ultralimit, in view of the discussion preceding this lemma,
\[ \begin{split}
    \lim_\omega \frac{1}{|F_n|^2} \sum_{x \in F_n} |{\bf C}_G(x) \cap F_n| &= \int_{G \times G} \mathbf{c}(x,y)
      \, d(\mu \otimes \mu)(x,y)
      \\&=  \int_G \mu \left( {\bf C}_G(x) \right) \, d\mu(x)
    \\&= \cp_\mu(G) \,.
\end{split}\]
This is the desired equality.
\end{proof}

By \cref{remark:sq}, in \cref{cor:multiplicativeFinite}, the only other possible contribution to the coefficient of $k^2$ is controlled by the probability that two independent $\mu$-random elements of $G$ have the same square. We define this quantity as the \emph{square incidence probability} and we set
\begin{equation} \label{eq: sq}
\sq_\mu(G) = \mathbb P_\mu(w;G), \quad \hbox{where } w(x,y)=x^2y^{-2} \,.
\end{equation}

\begin{corx}\label{cor:multiplicativeGroupMeasure}
    Let $G$ be an infinite group, let $(F_n)$ be a symmetric filtration series in $G$, let $k: \NN \to \NN$ be a non-decreasing function with $k \ll \sqrt{|F_n|}$, and let $\A$ be a random $k$-subset of $F_n$. Then
    \[ \E[\A,\A] = (1 + \cp_\mu(G) + \sq_\mu(G))k^2 + o(k^2) \,.\]
    Moreover, if $k \ll \sqrt[3]{|F_n|}$,
    \[ \E[\A,\A] = (1 + \cp_\mu(G) + \sq_\mu(G))k^2 - (\cp_\mu(G) + \sq_\mu(G))k + o(k) \,.\]
\end{corx}
\begin{proof}
    \cref{lemma:cP} implies that the ratio
    $|Q_3^{(1)}|/|F_n|^2 $ converges to $\cp_\mu(G)$. Meanwhile, the ultralimit of $ |Q_4^{(1)}| / |F_n|^2$ is, by definition,  $\sq_\mu(G)$. From these facts, the result follows at once following the same computations as in the proof of \cref{thm:multiplicative}.
\end{proof}

\subsection{Amenable groups} \label{subsec: 4-infinite-amenable}
We now  focus on the case in which $G$ is amenable, and $(F_n)$ is an actual F{\o}lner sequence. By definition, the mean obtained by exhaustion, and the associated finitely additive measure, are left-invariant. In this situation, the commuting probability is well-understood: the value of $\cp_\mu(G)$ is independent of the particular F{\o}lner sequence and of the subsequence along which the weak-$*$ limit is taken. Moreover, this invariant is positive if, and only if, $G$ is finite-by-abelian-by-finite  (see \cite[Theorem 1.14 and Proposition 1.15]{Tointon_commuting}), or, provided that $G$ is finitely generated, virtually abelian.

Returning to \Cref{example: heissenberg}, note that the Heisenberg group is finitely generated and amenable but not virtually abelian, and so $\cp(G)=0$, independently of the particular choice of $\mu$. Moreover, the squaring map $x\mapsto x^2$ is injective, so that
\[ |Q_{4}^{(1)}|=0 \quad \hbox{and} \quad |Q_{4}^{(2)}|=|F_n|(|F_n|-1)\]
for every finite filtration series. Substituting these values into \Cref{cor:multiplicativeGroupMeasure} yields the desired limit.

Our aim is now to prove parallel results for the square incidence probability. We begin with the following general lemma, which holds for every word $w$. (The following proof is inspired by \cite[Proposition 2.3]{AMV17}.)
\begin{lemma}
\label{lemma:wordQuot}
Let $G$ be an amenable group with a left-invariant measure $\mu$, and let $N \unlhd G$ be a normal subgroup of finite index. Then,
\[ \P(w;\, G/N) \geq \P_\mu(w;\, G).\]
\end{lemma}
\begin{proof}
For notational brevity, we set $p(G)= \P_\mu(w \, ; \, G)$ and $p(G/N)=\P(w \, ; \, G/N)$.
Let $W$ be the set of tuples in $G^k$ satisfying $w$, {\it i.e.},
\[
W=\{(g_1,\dots,g_k)\in G^k \mid w(g_1,\dots,g_k)=1\} \,,
\]
and let $(F_n)$ be the F{\o}lner sequence and $\omega$ the non-principal ultrafilter on $\N$ realising $\mu$. Assume by contradiction that $p(G/N) < p(G)$. Thus, there exist $\epsilon > 0$ and $R  \in \omega$ such that, for all $n \in R$,
\[ p(G/N) + \epsilon < \frac{|W \cap F_n^{k}|}{|F_n|^k} \,.\]
Moreover, since $\mu$ is left-invariant, $\mu(gN) = 1 / |G:N|$ for every $g \in G$. In particular, fixed $g \in G$, for every $\delta > 0$, there exists $S_g \in \omega$ such that, for all $n \in S_g$,
\begin{proofequation}\label{eq:DIOPORCO}
    \frac{|g N \cap F_n|}{|F_n|} \leq \frac{1}{|G:N|} + \delta \,.
\end{proofequation}
Choose a left-transversal $T$ for $N$ in $G$. Since there are finitely many representatives of cosets, the intersection
\[
S= R \cap \bigcap_{t \in T} S_t 
\]
is not empty and it belongs to $\omega$. Suppose that $\delta$ is small enough so that
\[\frac{\epsilon}{p(G/N) (2^k -1) |G:N|} > \delta,  \quad \hbox{and} \quad \frac{1}{ |G:N|} > \delta \,.\]
We compute
\[
\begin{split}
   p( G/N) + \epsilon & \le  \frac{|W \cap F_n^k|}{|F_n|^k}
   \\ &\le \frac{1}{|F_n|^k}\sum_{\substack{{\bf h } \in (G/N)^k\\ w({\bf h})=1}}
      \# \{\mathbf{g}\in F_n^k \mid \ \mathbf{g}\equiv {\bf h}\ \mod{N^k} \}
    \\&\le \frac{1}{|F_n|^k}\sum_{\substack{{\bf h}\in (G/N)^k\\ w({\bf h})=1}}
      \prod_{i=1}^k |h_i \, N\cap F_n|
    \\&\le \frac{1}{|F_n|^k} \max_{t \in T }{|t N \cap F_n|^k} \sum_{\substack{{\bf h}\in (G/N)^k\\ w({\bf h})=1}} 1
    \\&= \frac{1}{|F_n|^k} \max_{t \in T }{|t N \cap F_n|^k} \cdot \#\{\mathbf{h}\in (G/N)^k \mid w(\mathbf{h})=1\}
    \,.
\end{split}
\]
By recalling that
\[p(G/N) = \frac{\#\{\mathbf{h}\in (G/N)^k \mid w(\mathbf{h})=1\}}{|G\!:\!N|^k} \,,\]
we obtain, also invoking \cref{eq:DIOPORCO},
\[\begin{split}
    p( G/N) + \epsilon &\le \frac{1}{|F_n|^k} \max_{t \in T }{|t N \cap F_n|^k} \cdot p(G/N) |G:N|^k
    \\& \le p(G/N)\,(1+|G:N|\delta)^k\,.
\end{split}\]
By the elementary estimate, for every $x\in (0,1 ) $, $(1+x)^k \leq 1+(2^k-1)x$ we conclude that 
\[p( G/N) + \epsilon \leq p(G/N)  (1+ (2^k -1)|G:N|\delta)\,.\]
This contradicts our choice of $\delta$, thus completing the proof.
\end{proof}

We now characterise the amenable groups for which the square incidence probability $\sq_\mu(G)$ is positive. To this end, we shall require the additional assumption that $G$ is residually finite.

\begin{lemma}\label{lemma: sq charac}
Let $G$ be a residually finite amenable group, and let $\mu$ be a left-invariant finitely additive probability measure on $G$. Suppose that $\sq_\mu(G)>0$. Then, $G$ is virtually abelian.
\end{lemma}
\begin{proof}
Suppose that $\sq_\mu(G)= \alpha > 0$. By \cref{lemma:wordQuot}, $\sq(G/N) \geq \alpha$ for every finite quotient of $G$. Hence, by \cref{remark:cp,remark:sq},
$$\alpha \leq \sq(G/N) = \frac{\varepsilon(G/N)}{|G/N|} \leq \frac{\kappa(G/N)}{|G/N|}  = \cp(G/N) \,,$$
and thus, there are at least $\alpha |G/N|$ commuting pairs in $G/N$. Recall that $G$ embeds in the profinite completion
\[\hat{G} = \varprojlim_{N \unlhd_{\text{fin}} G} G/N \,.\]
Accordingly, in $\hat{G}$ with the Haar measure $\mu_H$, we have that
\[ \mu_H \otimes \mu_H \left(\left\{ (x,y) \in \hat{G}^2 \mid xy = yx \right\} \right) \geq \alpha \,.\]
Thus, by \cite[Theorem 1.1]{HR12}, $\hat{G}$ is virtually abelian. If we denote by $\hat{A} \unlhd \hat{G}$ the finite index abelian subgroup of $\hat{G}$, then $A = \hat{A} \cap G $ is a finite index abelian subgroup of $G$, as desired.
\end{proof}

We next show that, when $G$ is virtually abelian, $\sq(G)=\sq_\mu(G)$ is independent of the choice of $\mu$.
Before turning to the proof, we recall that any left-invariant measure $\mu$ on $G$ induces a left-invariant measure $\mu_H$ on a finite index subgroup $H\leq G$ by
\[
\mu_H(X)=\frac{\mu(X)}{\mu(H)}=|G:H|\,\mu(X),
\quad X\subseteq H.
\]
Accordingly, for every bounded $f\in \ell^\infty(G)$, we have
\begin{proofequation}\label{eq: integral}
\int_G f(g)\,d\mu(g)
=\frac{1}{|G:H|}\sum_{t \in T}\int_H f(s(t)\,a)\, d\mu_H(a),
\end{proofequation}
where $T$ is a left transversal for $H$ in $G$, and $s\colon T\to G$ is a section map.

\begin{lemma}
Let $G$ be a virtually abelian group. Then $\sq_\mu(G)$ is independent of the left-invariant measure $\mu$ on $G$.
\end{lemma}
\begin{proof}
Let 
\[S = \{ (x,y) \in G^2 \mid x^2 = y^2\} \,,\]
and let $\mathbf{1}_S$ be its characteristic function. Then,
\[
\sq_\mu(G) = \int_{G\times G} \mathbf{1}_S(x,y)\, d(\mu \otimes \mu)(x,y).
\]
Let $A \unlhd G$ be a finite index normal abelian subgroup, and denote by $Q$ the quotient group $G/A$.
Fix a section $s \colon Q \rightarrow G$ with $s(1)=1$.
Using the conjugation action of $Q$ on $A$ via $s$, for each $q \in Q$ we define the automorphisms
$f_q \in \Aut(A)$ by $f_q(a)= a^{s(q)} a$.
We denote by $\delta_q$ the diagonal $2$-cocycles, \emph{i.e.}, $\delta_q = s(q^2)^{-1} s(q)^2 \in A$, so that
for every $q\in Q$ and $a\in A$,
\[
(s(q)a)^2 = s(q^2)\delta_q f_q(a).
\]
Hence, we observe that $(s(q)a)^2 = (s(r)b)^2$ if and only if
\[
q^2 = r^2 \text{ in } Q \quad \text{and}\quad \delta_q f_q(a)= \delta_r f_r(b) \text{ in } A.
\]
Now, define, for each $q, r \in Q$ the group homomorphism
\[
F_{q,\, r} \colon A \times A \rightarrow A, \quad (a,b) \mapsto f_q(a)f_r(b)^{-1}.
\]
Provided that $q^2 = r^2$, the pairs with equal square are given by the fibers of $\delta_q\delta_r^{-1}$ via $F_{q,r}$. Explicitly, consider
\[
\mathcal{S}(Q)= \{ (q, r) \in Q^2 \mid q^2 = r^2 \text{ and } s(q)A \cap s(r)A \neq \varnothing\};
\]
then, whenever $(q, r) \in \mathcal{S}(Q) $, there exists $(\tilde{a}, \tilde{b}) \in A^2$ such that $F_{q, \, r}(\tilde{a}, \tilde{b}) =\delta_q\delta_r^{-1} $. Therefore, applying \cref{eq: integral},
\begin{equation*}
\begin{split}
\sq_\mu(G)
&=\frac{1}{|G:A|^{2}}
\sum_{\substack{(q, r) \in \mathcal{S}(Q)}}
(\mu_A\otimes\mu_A)\bigl(\{(a,b)\in A^2 \mid F_{q,r}(a,b)=\delta_q\delta_r^{-1}\}\bigr) \\
&=\frac{1}{|G:A|^{2}}
\sum_{\substack{(q, r) \in \mathcal{S}(Q)}}
(\mu_A\otimes\mu_A)\bigl((\tilde{a}, \tilde{b})\,\ker(F_{q,r})\bigr) \\
&=\frac{1}{|G:A|^{2}}
\sum_{\substack{(q, r) \in \mathcal{S}(Q)}}
(\mu_A\otimes\mu_A)\bigl(\ker(F_{q,r})\bigr)
\\&=\frac{1}{|G:A|^{2}}
\sum_{\substack{(q, r) \in \mathcal{S}(Q)}}
\frac{1}{|A^2 : \ker(F_{q,r})|} \,.
\end{split}
\end{equation*}
As the last expression is independent of the measure $\mu$, the proof is complete.
\end{proof}

Collecting the above results, we obtain the following specialization of \cref{thm:multiplicative} and \cref{cor:multiplicativeGroupMeasure}.
\begin{corx}\label{cor:multiplicativeamenable}
     Let $G$ be a finitely generated amenable infinite group, let $(F_n) \subseteq G$ be a symmetric F\o lner sequence of $G$, let $k \colon \mathbb{N} \to \mathbb{N}$ be a nondecreasing function, and let $\A$ be a $k$-subset of $F_n$ sampled uniformly at random.
    \begin{enumerate}[$(a)$]
    \item Suppose that $k = \alpha |F_n| + o(|F_n|)$, for some constant $\alpha \in [0,1]$. Then,
    \[ \alpha^4 \le \E[\eta_n(\A,\A)] \le \alpha^4 + \alpha^3 \,.\]
    
    \item Suppose that $G$ is virtually abelian, and that $k \ll \sqrt{|F_n|}$. Then,
  \[  \E[E(\A ,\A )] =  k^2 \left(1 + \cp(G) + \sq(G) \right) + o(k^2) \,. \]
    \item Suppose that $G$ is not virtually abelian, and that $k \ll \sqrt{|F_n|}$. Then,
    \[  k^2 \le \E[E(\A ,\A )] \le  2k^2 + o(k^2)  \,. \]
    Moreover, when $G$ is residually finite as well,
    \[  \E[E(\A ,\A)] =  k^{2} + o(k^2) \,. \]
    \end{enumerate}
\end{corx}

\begin{proof}
    Item $(a)$ follows from a straightforward computation from \cref{thm:multiplicative}.
    
    To prove items $(b)$ and $(c)$, we observe that \cref{lemma:cP} implies that the ratio
    $|Q_3^{(1)}|/|F_n|^2 $ converges to $\cp(G)$, and  by \cite[Theorem 1.17]{Tointon_commuting}, $\cp(G) =0$ when $G$ is not virtually abelian.
    Furthermore, the ultralimit of $ |Q_4^{(1)}| / |F_n|^2$ is, by definition,  $\sq_\mu(G)$. By \cref{lemma: sq charac}, $\sq(G) =\sq_\mu(G)$ is independent of $\mu$ when $G$ is residually finite, and $\sq(G)=0$ when $G$ is further not virtually abelian. From here on, the computations are identical to those in the proof of \cref{thm:multiplicative}.
\end{proof}

Observe that, the use of residually finiteness in \cref{lemma: sq charac} seems like a technicality, rather than a structural requirement. Hence, we propose the following question to strengthen \cref{cor:multiplicativeamenable}~$(c)$.
\begin{question}\label{quest:lastDilemma}
    Can we drop the hypothesis of residually finiteness in \cref{lemma: sq charac} and \cref{cor:multiplicativeamenable}~$(c)$?
\end{question}

Finally, we record an application of the preceding asymptotic bounds. A group $G$ has the \emph{small squaring property} if, for every integer $k\ge 2$, and for every $k$-subset $A \subseteq G$, $A^{*2}$ has size strictly less than $k^2$. This notion was introduced by Freiman in \cite{Freiman1981} for $k=2$, and later developed further in \cite{BerkovichFreimanPraeger1991,BrailovskyFreiman1981, FreimanHerzogLongobardiMajPlagneStanchescu2017}. In view of \cref{remark:previousParagraph}, our method provides a novel proof of a result of Neumann (see \cite[Theorem 1]{HerzogLongobardiMaj1993}), provided that the group can, in some sense, be approximate by finite groups.

\begin{corx}\label{thm: small squaring property}
Let $G$ be a residually finite amenable group. If $G$ satisfies the small-squaring property for some integer $k$, then $G$ is finite-by-abelian-by-finite.
\end{corx}

\begin{proof}
Our proof consists in verifying the veracity of the counter-nominal statement. Suppose that $G$ is a residually finite amenable group that is not finite-by-abelian-finite. By combining \cref{lemma: sq charac} and \cite[Theorem 1.14]{Tointon_commuting}, we obtain that, for every left-invariant measure $\mu$, $\sq_\mu(G)= \cp_\mu(G)=0$. Hence, for a random $k$-subsets $\A$, by \cref{cor:multiplicativeamenable},
\[\lim_n \E\left[ E(\A,  \A) \right] = k^2 \,.\]
Therefore, by Jensen's inequality and \cref{lemma:CauchySchwartz},
\[\lim_n \E\left[|\A^{\ast 2}| \right] = k^2 \,.\]
Hence, for every positive integer $k\ge 2$, there exists $N(k)$, such that for $n \ge N(k)$,
\[ \E\left[|\A^{\ast 2}| \right] \ge k^2 -\frac{1}{2} \,.\]
As the random variable is supported only on integers, this yields the existence of a set $A_k$ such that $|A_k^{\ast 2}| = k^2$, and thus $G$ does not have the small squaring property, as desired. 
\end{proof}

In analogy with Corollary \ref{quest:lastDilemma}, we expect this result to extend to the class of all amenable groups. In view of \cite[Theorem 1.14]{Tointon_commuting}, a natural approach to achieving this would be to provide an affirmative answer to the following question.
\begin{question}
    Is it true that, for every amenable group, $\sq_\mu(G) \leq \cp_\mu(G)$ ?
\end{question}

\subsection{Finitely generated groups} \label{subsec: 4-infinite-fingen}
We now turn our attention to finitely generated groups with the word metric. Let $S$ be a finite generating set for $G$ and let $(F_n)_{n \in \mathbb{N}} = (\B_S(n))_{n \in \mathbb{N}}$ be the filtration series by balls around the identity in the word metric induced by $S$. By \cref{cor:multiplicativeGroupMeasure}, in the regime $k \ll \sqrt{|F_n|}$,
\[ \E[E(\A ,\A )]
= k^2\bigl(1+\cp_\mu(G)+\sq_\mu(G)\bigr)+ o(k^2)  \,.
\]
Note that, \emph{a priori}, both the square-incidence probability and the commuting probability may depend on the choice of the generating set $S$, and consequently the finite exhaustion measure $\mu$.
The notion of commuting probability was introduced in this setting in \cite{AMV17}, although the authors consider subsequences along which the {\it limes superior} is attained. It remains an open problem whether the resulting invariant $\cp_S(G)$ (we use the subscript $S$ to emphasise the dependence on the generating set rather than on the measure) is independent of the choice of $S$. It has further been conjectured that $\cp_S(G)$ is non-zero if, and only if, $G$ is virtually abelian (see \cite[Conjecture 1.6]{AMV17}).

Recall that groups of subexponential growth are amenable, in which case there is a symmetric F{\o}lner sequence consisting of balls. Hence, the values of $\cp_\mu(G)$ and $\sq_\mu(G)$ are independent of the choice of $\mu$, by \cite[Corollary~1.18]{Tointon_commuting} and \cref{lemma: sq charac}, respectively. Furthermore, we suspect that, for every group $G$ of exponential growth, both $\sq_S(G)$ and $\cp_S(G)$ are zero. We thus propose the following conjecture.
\begin{conjecture}\label{conj:EarthboundZero}
Let $G$ be a finitely generated group of exponential growth. Let $k \colon \mathbb{N} \to \mathbb{N}$ be a nondecreasing function with $k \ll \sqrt{|\B_S(n)|}$, and let $\A$ be a $k$-subset of $F_n$ sampled uniformly at random in the word-balls $\B_S(n)$. Then,
\[
\lim_{n} \E[E(\A,\A)] = k^2 \,.
\]
\end{conjecture}

We prove the veracity of this conjecture for certain classes of groups for which the commuting probability is already known.
\begin{example}[Word hyperbolic groups]
\label{ex: AA hyperbolic}
Let $G$ be a nonelementary hyperbolic group, and let $S$ be a finite generating set of $G$. Regardless of the generating set, $\cp_S(G) =0 $ (see \cite[Theorem 1.7]{AMV17}).

For the square incidence probability, let $\operatorname{Tor}(G)$ denote the set of torsion elements of $G$, and recall that 
\[ \# \{ (x,y) \in \B_S(N)^2 \mid x^2 = y^2  \} = \sum_{g \in G} r_n(g)^2  = \sum_{g \in \operatorname{Tor}(G)} r(g)^2 +\sum_{g \not\in \operatorname{Tor}(G)} r_n(g)^2 \,.\]
For the first summand,
\[\sum_{g \in \operatorname{Tor}(G)} r_n(g)^2  \leq |\operatorname{Tor}(G)|^2 \,, \]
and \cite[Theorem 1.1]{Dani} gives that
$$\limsup_n \frac{ |\operatorname{Tor}(G)|}{|\B_S(n)| } = 0 \,.$$
For the second summand, recall that centralizers of elements of infinite order in hyperbolic groups are finite-by-$\Z$ or finite-by-$(C_2 \ast C_2)$ (see \cite[Corollary III.$\Gamma$.3.10]{BrHa}), and there is a bound $C > 0$ on the size of finite subgroups of $G$ depending only on the hyperbolicity constant (see \cite[Theorem III.$\Gamma$.3.2]{BrHa}). Fix an element $x \in G$ of infinite order, and let $N \unlhd {\bf C}_G(x^2)$ be the finite normal subgroup such that ${\bf C}_G(x^2) / N $ is isomorphic to either $\Z$ or $C_2 \ast C_2$. Note that two elements in $\Z$ or $C_2 \ast C_2$ share the same square root if, and only if, they are both torsion. As a consequence, if $x^2 = y^2$ for an element $y \in G$, then $x \equiv y \pmod{N}$, and thus $y = m x$ for $m \in N$. Therefore, $r_n(x^2) \leq |N| \leq C$. Thus,
\[ 0 \le \limsup_n \frac{1}{|\B_S(n)|^2} \, \sum_{g \not\in \operatorname{Tor}(G)} r_n(g)^2 \leq \limsup_n \frac{ |\B_S(n)| C^2}{|\B_S(n)|^2 } = 0 \,.\]
Hence, $\sq_S(G)=0$.
\end{example}

\begin{example}[Right-angled Artin groups]
\label{ex: AA RAAG}
Let $G=G(\Gamma)$ be an infinite right-angled Artin group on a finite graph
$\Gamma$, equipped with the canonical generating set
$S=V(\Gamma)\cup V(\Gamma)^{-1}$. On the one hand, \cite[Theorem 6]{Valiunas19} yields $\cp_S(G)=0$. On the other hand, right-angled Artin groups are bi-orderable (see \cite{DuchampThibon1992}). As a consequence, the roots are unique, {\it i.e.}, if $g^m=h^m$ for some $m\ge 1$, then $g=h$. Hence, the squaring map $g \mapsto g^2$ is injective, and $\sq_S(G) =0$ (for every generating set).
\end{example}

\begin{example}[Generalised lamplighter with infinite cyclic top group] \label{ex: AA lamp}
Suppose that $G $ is the restricted wreath product $ H \wr \Z = H^\Z \rtimes \Z$, where $H$ is a finitely generated group. On the one hand, $\cp_S(G)=0$ for any generating set (see \cite[Theorem 2]{Cox18} and \cite[Theorem A]{A3}).

On the other hand, suppose that $x = (h, t^n)$ and $y = (k, t^m)$ share the same square but they are distinct. We compute  
\[ x^ 2= (h h^{t^n}, \, t^{2n}) = (k k^{t^m}, \, t^{2m}) = y^2 \,.\]
Since the squaring map is injective in $\Z$, $n=m$. Write $h= (h_i)_{i \in \Z} $ and $k= (k_i)_{i \in \Z}$. The previous equality implies that 
\begin{proofequation}
\label{eq: lamp}
h_{i} h_{i-n} = k_i k_{i-n}
\end{proofequation}
Suppose that there exists $i_0 \in \Z$ such that $h_{i_0} \neq k_{i_0}$, and let $i_0$ be minimal with that property. Thus, unless $n=0$, $h_{i_0 - n} =k_{i_0 -n}$. Substituting  in \cref{eq: lamp}, we get the contradiction $h_{i_0} = k_{i_0}$. Therefore, $n=0$, and $x, y$  belong both to the base group $ H^{\Z}$. Hence, upon splitting the set of pairs sharing the same square in those whose entries are distinct and those whose are not, we get
\[\frac{\bigl|\{(x,y)\in \B_S(n)^2 \mid x^2 = y^2\}\bigr|}{|\B_S(n)|^2} \le
\frac{|\B_S(n)| + \bigl|\{(x,y)\in \B_S(n)^2 \mid x \neq y, x^2 = y^2\}\bigr|}{|\B_S(n)|^2} \,.\]
Our previous considerations, together with \cite[Theorem B]{A3}, yield
\[ \lim_n \frac{\bigl|\{(x,y)\in \B_S(n)^2 \mid x^2 = y^2\}\bigr|}{|\B_S(n)|^2} \le \lim_n  \frac{1}{|\B_S(n)|}
+ \frac{|H^{\mathbb{Z}} \cap \B_S(n)|^2}{|\B_S(n)|^2} = 0 \,. \]
Therefore, $\sq_S(G)=0$.
\end{example}

\section{Multiplicative energy of a set with its set of inverses}
Another natural substitute for $\Delta$ rather than $A$ itself is the set of its inverses, $A^{-1}$, as natural cancellation phenomena occur. Our control in this situation is greater: rather than dealing with properties of specific pairs or triplets of elements in $F_n$, we only need to compute how many involutions appear in each set of the filtration. This will be explicit in \cref{remark:potatoes}. Our main result is blind to this fact, as it is a straightforward adaptation of \cref{thm:action}.
\begin{thmx}\label{thm:multiplicativeInverse}
    Let $G$ be a discrete group, let $(F_n) \subseteq G$ be a symmetric filtration series of $G$, let $k \colon \mathbb{N} \to \mathbb{N}$ be a nondecreasing function, and let $\A$ be a $k$-subset of $F_n$ sampled uniformly at random. Then,
    \[\begin{split}
        \E[E_n(\A,\A^{-1})] &\le \frac{k^{\underline 4}}{|F_n|-3} + \frac{k^{\underline 3}}{(|F_n|-1)^{\underline 2}}\left( 2|F_n| - 1 \right) + 3k^2 - 2k \,,
    \end{split}\]
    and
    \[\begin{split}
        \E[E_n(\A,\A^{-1})] &\ge \frac{k^{\underline 4}}{(|F_n|-1)^{\underline 3}} \left(|F_n|^2 - 4|F_n| +2 \right) +  2k^2 - k \,.
    \end{split}\]
\end{thmx}
\begin{proof}
    This proof is essentially the same as that of \cref{thm:multiplicative}. Hence, we just stress where the differences arise. To start with, the usual double counting tells us that 
    \[\E\left[ E(\A,\A^{-1}) \right] = {|F_n| \choose k}^{-1} \sum_{(a,b,c)\in F_n^3} \# \left\{ A \in \Pa_k(F_n)  \middle|  (a,b,c,ba^{-1}c) \in A^4 \right\} \,.\]
    The partition we used previously contains some empty sets. Hence, we can just use these seven sets:
    \begin{align*}
        Q_1^{(1)} &= \{(a,b,c)\in F_n^3\mid \# \{a,b,c\} =3,\, ba^{-1}c= a \}\,,\\
        Q_1^{(2)} &= \{(a,b,c)\in F_n^3\mid \# \{a,b,c\} =3,\, ba^{-1}c\neq a \}\,,\\
        Q_2 &= \{(a,b,a)\in F_n^3\mid a\neq b \}\,,\\
        Q_3^{(1)} &= \{(a,b,b)\in F_n^3\mid a\neq b,\, ba^{-1}b = a \}\,,\\
        Q_3^{(2)} &= \{(a,b,b)\in F_n^3\mid a\neq b,\, ba^{-1}b \neq a \}\,,\\
        Q_4 &= \{(a,a,c)\in F_n^3\mid a\neq c \}\,,\\
        Q_5 &= \{(a,a,a) \in F_n^3\}\,.
    \end{align*}

    Once again, it is easy to see that
    \[ |Q_2|=|Q_4|=|F_n|(|F_n|-1), \quad \hbox{and} \quad |Q_5|=|F_n| \,.\]
    We now split the discussion for the pairs $Q_1^{(1)}$ and $Q_1^{(2)}$, and $Q_3^{(1)}$ and $Q_3^{(2)}$.

    We start from the former pair. We denote by $I$ the set of all involutions (including the identity) in $G$. For every $y \in F_n^{*2} \setminus I$, we define an injective map
    \[\psi_y :\,  F_n  \to F_n \times y F_n \times y^{-1} F_n , \, x \mapsto (x,yx, y^{-1} x) \,.\]
    We claim that
    \begin{equation*}
        \left( \bigcup_{y \in F_n^{*2} \setminus I } \psi_y(F_n)  \right) \cap F_n^3  = Q_1^{(1)} \,.
    \end{equation*}
    On the one hand, if $y \in F_n^{\ast2} \setminus I$ and $(x,yx, y^{-1} x) \in F_n^3$, then
    \[ (yx) x^{-1} (y^{-1} x) = x \,.\]
    Since $y^2 \neq 1$, we see that $x,\ yx$ and $y^{-1} x$ are pairwise distinct, so $(x,yx, y^{-1} x) \in Q_1^{(1)}$.
    On the other hand, for every $(a,b,c) \in Q_1^{(1)}$, let $x = a \in F_n$ and $y = ba^{-1} \in F_n^{*2}$. Then, $b = yx$ and
    \[ c = ab^{-1}a = x x^{-1}y^{-1}x = y^{-1} x \,. \]
    If $y^2=1$, then $ba^{-1}=ab^{-1}$, and thus $c = b$, against the choice of $(a,b,c)$. Therefore, $(a,b,c)= \psi_y(x)$ for $y \in F_n^{\ast2} \setminus I$ and $x \in F_n$. This proves the claim.
    Observe that $\psi_y(x) \in F_n^3$ if, and only if, $x \in F_n \cap y F_n \cap y^{-1} F_n$, implying that
\[ |\psi_y(F_n) \cap F_n^3| = |F_n \cap y F_n \cap y^{-1} F_n | \leq |F_n \cap y F_n| \,.\] 
Therefore, using \Cref{lemma:andoniRevenge},
\begin{equation*}
    0 \leq |Q_1^{(1)}| \leq \sum_{y \in F_n^{*2} \setminus I } |F_n \cap y F_n| \leq |F_n|^{2} \,. 
\end{equation*}
As $ |Q_1^{(1)}| + |Q_1^{(2)}| = |F_n|(|F_n| -1)(|F_n|-2)$, we obtain 
\[|F_n| \left( |F_n|^2 -4 |F_n| +2  \right) \leq |Q_1^{(2)}| \leq |F_n| (|F_n| -1)(|F_n|-2) \,. \] 
    
For the pair $ Q_3^{(1)}$ and $ Q_3^{(2)}$, for the current proof we take the trivial bounds
\[ 0 \le  |Q_3^{(1)}| \le |F_n| (|F_n| -1) \quad \hbox{and} \quad 0 \le  |Q_3^{(2)}| \le |F_n| (|F_n| -1) \,.\]

To conclude, we note that
\begin{align}
\label{eq: main AAinv}
    \E\left[ E(\A, \A^{-1}) \right] =& {|F_n| \choose k}^{-1}\left[ {|F_n|-4\choose k-4} |Q_1^{(2)}| + {|F_n|-3\choose k-3} (|Q_1^{(1)}|+ |Q_3^{(2)}| ) + \right. \\
    &\left.+ {|F_n|-2\choose k-2} (|Q_2| + |Q_3^{(1)}| + |Q_4|) + {|F_n|-1\choose k-1} |Q_5| \right] \nonumber \,.
\end{align} 
The remainder of the proof is just computations.
\end{proof}

The asymptotic bounds for $E_n(\A,\A^{-1})$ that we can extract from \cref{thm:multiplicativeInverse} are already sharper than those that \cref{thm:multiplicative} gives for $E_n(\A,\A)$. We have that, if $k \ll \sqrt{|F_n|}$, 
\[ 2 k^2 \le \E[E_n(\A,\A)] - O\left(\frac{k^4}{|F_n|}\right) \le 3k^2  \,,\]
and, if $k \ll \sqrt[3]{|F_n|}$,
\[ 2 k^2 - k \le \E[E_n(\A,\A)] - O\left(\frac{k^4}{|F_n|}\right) \le 3k^2 -2k  \,,\]

 As for \cref{thm:multiplicative}, these asymptotic bounds are sharp, and thus they cannot be improved without adding extra assumptions to the model. The upper bound is attained by elementary abelian $2$-groups for any exhaustion, because $Q_3^{(2)}$ is empty, and thus $|Q_3^{(1)}| = |F_n|(|F_n|-1)$. On the other hand, every element in $Q_3^{(1)}$ has order dividing $2$, hence $Q_3^{(1)}$ is empty whenever $G$ is torsion-free, and the lower bound is attained. 
 
As for the corollaries of \cref{thm:multiplicative} most of the heavy lifting has been done by understanding the pair $|Q_4^{(1)}|$ and $|Q_4^{(2)}|$, our improvements of \cref{thm:multiplicativeInverse} will mainly come from understanding $|Q_3^{(1)}|$ and $|Q_3^{(2)}|$. Here, the cardinality of the intersections of the set $I$, containing all involutions of $G$, and the sets of the filtration series will play a central role.
\begin{remark}\label{remark:potatoes}    
    For each $x \in F_n$  we define the injective map
    \[\phi_x: F_nx^{-1} \cap \left( I \setminus \{1\} \right)  \to  F_n^3 ,\quad y \mapsto (x, yx, yx) \,.\]
    We claim that 
    \[ \left( \bigcup_{x \in F_n} \phi_x(F_nx^{-1} \cap \left( I \setminus \{1\} \right)) \right) \cap F_n^3 = Q_3^{(1)} \,.\]
    On the one hand, if $(a,b,b) \in Q_3^{(1)}$, then we set $x = a \in F_n$ and $y = b a^{-1} \in F_n a^{-1}$, so that $(a,b,b) = \phi_x(y)$. Since
    \[x = a= b a^{-1} b = yx x^{-1} yx = y^{2}x ,\] 
    we have $y \in I$. Moreover, since $a \neq ya$, $y \neq 1$.
    On the other hand, if $x  \in  F_n$ and $ y \in F_n x^{-1} \cap I \setminus \{1\}$, then
    \[ (yx) x^{-1} (y x) = y^2 x = x \,.\]
    Since $y \neq 1$, we see that $x \neq yx$, and thus $\phi_x(y) \in Q_3 ^{(1)}$.
    We observe that 
    \[ |\phi_x(F_n x^{-1} \cap \left(I \setminus \{1\}\right)) \cap F_n^3| = |F_n x^{-1} \cap \left(I \setminus \{1\}\right)  | \,. \]
    Hence, noting that $|Q_3^{(1)}| + |Q_3^{(2)}| = |F_n| (|F_n|-1)$, we have that 
    \begin{equation*}
    0 \leq |Q_3^{(1)}| \leq \sum_{x \in F_n} |F_n x^{-1} \cap \left(I \setminus \{1\}\right) |  \,,\end{equation*}
    and
    \begin{equation*}\label{eq:q2}
    |F_n|(|F_n| - 1) - \sum_{x \in F_n} |F_n x^{-1} \cap \left(I \setminus \{1\}\right) | \leq  |Q_3^{(2)}| \leq |F_n|(|F_n|-1) \,.
    \end{equation*}
\end{remark}
    We stress that we have not included this consideration in the proof of \cref{thm:multiplicativeInverse}, because it might be worse than the trivial bound $|F_n|(|F_n|-1)$. In fact,
    \[ \sum_{x \in F_n} |F_n x^{-1} \cap \left(I \setminus \{1\}\right) | \le |F_n| |F_n^{\ast 2} \cap \left(I \setminus \{1\}\right) | \,.\]
    Although this last bound will be useful in \cref{subsec:inf2}, we are unable, in full generality, to rule out the possibility that $|F_n^{*2}\cap (I\setminus\{1\})|$ grows superlinearly in $|F_n|$. (Notwithstanding, $|F_n^{*2}|$ itself may have superlinear growth in $|F_n|$. For instance, this occurs for the ball filtration in groups of exponential growth.)

\subsection{Finite groups}
We now go back to focusing on $G$ being finite. We denote by $\iota(G)$ the number of involutions of $G$  (including the identity). It is well known that in a finite group, the number of involutions can be expressed as
\[ \iota(G) = r(1) = \frac{1}{|G|}\sum_{\chi \in\mathrm{Irr}(G)} \left(\chi(1) \sum_{g\in G} \chi(g^2) \right) \,,\]
(This is immediate from \cref{lemma:squareRoots}, and a similar proof can be found in \cite[Corollary~4.6]{Isaacs}.)
\begin{corx}\label{cor:multiplicativeInverseFinite}
     Let $G$ be a finite group, let $k \in \mathbb{N}$ be a positive integer, and let $\A$ be a $k$-subset of $G$ sampled uniformly at random. Suppose that $\iota$ is the number of involutions of $G$ (including the identity). Then
    \[\begin{split}
        \E[E(\A,\A^{-1})] &= \frac{k^{\underline 4}}{(|G|-1)^{\underline 3}} ( |G|^2 - 4|G| +2 + \iota(G)) + \frac{k^{\underline 3}}{(|G|-1)^{\underline 2}} 2 \left( |G| - \iota(G) \right)
        \\&\, + \frac{k^{\underline 2}}{(|G|-1)} \left(\iota(G) - 1\right) + 2k^2-k \,.
    \end{split}\]
\end{corx}
\begin{proof}[Proof]
    As $G$ is finite, we can assume that, for every $n$, $F_n = G$, and, hence, $G^{*2}$ is nothing but $G$. Under this assumption, we can explicitly compute $|Q_1^{(1)}|$, $|Q_1^{(2)}|$, $|Q_3^{(1)}|$ and $|Q_3^{(2)}|$. Substituting these values into the usual formula yields the desired statement.
    
    Explicitly, while computing the cardinality of $|Q_1^{(1)}|$, we find
    \[\begin{split}
        |Q_1^{(1)}| &= \sum_{y \in G \setminus I} \left| \psi_y(G) \cap G^3 \right|
        \\&= \sum_{y \in G \setminus I} \left| G \cap yG \cap y^{-1}G \right|
        \\&=  \sum_{y \in G \setminus I} |G|
        \\&= |G| |G \setminus I|
        \\&= |G|(|G| - |I|) \,.
    \end{split}\]
    Therefore, recalling that $|Q_1^{(1)}| + |Q_1^{(2)}| = |G|(|G|-1)(|G|-2)$,
    \[ |Q_1^{(2)}| = |G| (|G|^2 - 4|G|  + |I| +2) \,.\]
    Similarly, by \cref{remark:potatoes},
    \[ |Q_3^{(1)}| = \left| \bigcup_{x\in G} \phi_x\ (I\setminus \{1\}) \right|= |G|(|I| - 1), \quad \hbox{and} \quad |Q_3^{(2)}| = |G|\left(|G| - |I| \right)\,. \qedhere\]
\end{proof}

\begin{example}
    We consider again $G = \mathrm{GL}_2(q)$.
    Suppose first that $q$ is an odd integer. Then $|G| =  (q + 1)q(q - 1)^2$, and, from a direct inspection of the character table of $G$, $\iota(G) = q^2 + q + 2$, and
\[
\begin{split}
\E[E(\A,\A^{-1})] & = \frac{k(1+k)(2+k)(3+k)\bigl(4 + q + q^2 - 4(q-1)^2 q(1+q) + (q-1)^4 q^2 (1+q)^2\bigr)}{(q-1)^2 q(1+q)\bigl((q-1)^2 q(1+q)-1\bigr)\bigl(1 + (q-1)^2 q(1+q)\bigr)} \\
&+ \frac{k(1+k)\bigl(1 + q + q^2\bigr)}{(q-1)^2 q (1+q)-1} + \frac{2k(1+k)(2+k)\bigl(q^4 - q^3- 2q^2  -2 \bigr)}{(q-1)^2q \bigl(q^5 - 2q^3 -1 \bigr)} + 2k^2 - k \,.
\end{split}\]

If $q$ is an even integer, then $|G| =  (q + 1)q(q - 1)^2$, $\iota(G) = q^2$, and 
\begin{align*}
\E[E(\A,\A^{-1})] &= \frac{k(1+k)(2+k)(3+k)\bigl(2 + q^2 - 4(-1+q)^2 q(1+q) + (q-1)^4 q^2 (1+q)^2\bigr)}{(q-1)^2 q (1+q)\bigl( (q-1)^2 q(1+q)-1\bigr)\bigl(1 + (q-1)^2 q(1+q)\bigr)} \\
&+ \frac{k(1+k)\bigl( q^2 - 1\bigr)}{(q-1)^2 q (1+q)-1} + \frac{2k(1+k)(2+k)\bigl(1 - 2q - q^2 + q^3\bigr)}{(q-1)^2 \bigl( q^5 - 2q^3-1\bigr)} + 2k^2 - k\,.
\end{align*}
\end{example}

Accordingly, for every sequence $(G_n)$ of finite groups, and for $k$ bounded, we have the asymptotic expression
\begin{equation}\label{eq:invProp}
    \E[E(\A,\A^{-1})] = \left( 2 + \frac{\iota(G_n)}{|G_n|} \right)k^2 - \left(1 + \frac{\iota(G_n)}{|G_n|} \right)k + O\left( \frac{\iota(G_n)}{|G_n|} \right) \,.
\end{equation}
Observe that $\iota(G_n)$ is zero for groups of odd order, and that $\lim_n \iota(G_n) = 0$ for sequences of finite simple groups $(G_n)$ of increasing order.

\subsection{Infinite groups}\label{subsec:inf2}
As for finite groups, by \cref{remark:potatoes}, the precise limit will as well depend on the proportion of involutions of the group. As in \cref{subsec:infinite}, if $G$ is a group endowed with a finitely additive measure, then the \emph{involution probability} of $G$ with respect to $\mu$ is
\begin{equation}
\label{eq: involution density}
\iota_\mu(G) = \P(G; \, \beta_2) =  \mu(\{ x \in G \mid x^2 = 1\}) \,,
\end{equation}
where $\beta_2(x)= x^2$ is the $2$-Burnside word.

We note that, by definition, the ratio $|Q_3^{(1)}|/|F_n|^2$ converges to $\mathbb P_\mu(G; \, w)$ for the $2$-variant word $ w(x,y)=(x^{-1}y)^2$ . A direct computation therefore yields
\[
\E[\A,\A]=(2+\mathbb P_\mu(G;w))k^2+o(k^2),
\]
for any measure.
Nonetheless, in keeping with the phenomena observed for finite groups, we conjecture that the expected value of $E_n(\A,\A^{-1})$ is governed by the involution probability.

\begin{question}\label{question:golden}
Let $G$ be an infinite group, let $(F_n)$ be a symmetric filtration in $G$, let $k\colon \mathbb N\to\mathbb N$ be a non-decreasing function with $k\ll \sqrt{|F_n|}$, and let $\A$ be a random $k$-subset of $F_n$. Is it true that
\[
\E[\A,\A^{-1}] = (2+\iota_\mu(G))k^2 + o(k^2)\,?
\]
\end{question}

Although we do not settle this in full generality, we answer the question affirmatively in several cases.

\subsection{Amenable groups} \label{subsec: 5-infinite-amenable}
Leveraging on the left-invariant measure, we characterise residually finite groups where the involution probability is strictly positive. We begin by recording an auxiliary lemma adapted to our framework. (The following proof is along the lines of \cite[Theorem 2]{LP}.)

\begin{lemma}
\label{lemma: involution charac}
Let $G$ be a residually finite amenable group with a left-invariant measure $\mu$. Then, $\iota_\mu(G) > 0$ if, and only if, there exists a finite index normal abelian subgroup $A \unlhd G$ and there exists $t \in G$ such that every element in $tA$ is an involution.
\end{lemma}
\begin{proof}
The converse implication is straightforward, because, using the left-invariance of the measure,
\[ \iota_\mu(G) \geq \mu(tA) = \mu(A) = \frac{1}{|G:A|} > 0 \,.\]

For the direct implication, we use that $G$ embeds in the profinite completion
\[\hat{G} = \varprojlim_{N \unlhd_{\text{fin}} G} G/N \,.\]
By Lemma~\ref{lemma:wordQuot}, $\iota(G/N) \geq \alpha$ for each finite quotient of $G$, and, hence, by \cite[Lemma 1]{Mann94}, then $\cp(G/N) > \alpha^2$, so there are at least $\alpha^2 |G/N|$ commuting pairs in $G/N$. Therefore, in $\hat{G}$ with the Haar measure $\mu_H$, we have
\[ \mu_H \otimes \mu_H (\{ (x,y) \in G^2 \mid xy = yx \}) \geq \alpha^2 \,.\]
Thus, by \cite[Theorem 1]{LP}, $\hat{G}$ is virtually abelian. Let $\hat{A} \unlhd \hat{G}$ be a finite index abelian normal subgroup of $\hat{A}$, so that $A = \hat{A} \cap G $ is a finite index abelian normal subgroup of $G$. Let $I$ be the set of all involutions of $G$, since $A$ has finite index in $G$ and $\mu(I) = \iota_\mu(G) > 0$, there exists $g \in G$ such that $\mu(I \cap g A) > 0$. Now, by eventually substituting $g$ by an element $gh \in I$, we can assume that $g$ itself is an involution. Therefore, 
\[ ga \in I \cap gA \quad \hbox{if, and only if,} \quad gaga= a^{g}a =1 \,.\]
Hence, if we define the subgroup
\begin{proofequation}\label{eq:gaga}
    H_g =  \left\{ a \in A \mid a^{g}a =1 \right\} \,,
\end{proofequation}
then $I \cap gA = g H_g$. Using the left-invariance of the Haar measure,
\[\mu_H(H_g) = \mu_H(gH_g) = \mu_H(I \cap gA) > 0 \,.\]
Therefore, $H_g$ has finite index in $G$, and every element in $gH_g$ is an involution, as desired.
\end{proof}

We now prove that, for a residually finite amenable group $G$, actually $\iota_\mu(G)$ does not depend on $\mu$. Hence, we will avoid writing the subscript in this context.
\begin{lemma}
\label{lemma: involution invariant}
Let $G$ be a residually finite amenable group with a left-invariant measure $\mu$. Then, $\iota_\mu(G)$ is independent of the measure $\mu$.
\end{lemma}
\begin{proof}
By \cref{lemma: involution charac}, if $G$ is not virtually abelian, then, for every left-invariant measure $\mu$, $\iota_\mu(G)=0$. Hence, we can assume that there exists a finite index abelian normal subgroup $A \unlhd G$. Consider a left transversal $T$ for $A$ in $G$, and the partition of the set $I$ of involutions given by
\[ I = \bigcup_{g \in T} I \cap g A \,.\]
Let $T_0$ be the set of $g\in T$ such that $I \cap gA$ is non-empty. Note that, for every $g\in T_0$, we can assume that $g$ is an involution itself. Moreover, let $H_g$ be the subgroup defined in \cref{eq:gaga}, so that
$ I \cap g A = gH_g $. It follows that 
\begin{proofequation}\label{eq:expressionIota}
    \iota_\mu(G)= \mu(I)  = \sum_{g \in T} \mu(I \cap g A) = \sum_{g \in T_0} \mu(gH_g) = \sum_{g \in T_0} \mu(H_g)  = \sum_{g \in T_0} \frac{1}{|G:H_g|}\,
\end{proofequation}
where the reciprocal of an infinite index is zero. As the last expression is independent of $\mu$, the proof is complete.
\end{proof}

\begin{lemma}
\label{lem: virt abelian}
Let $G$ be a virtually abelian group, let $(F_n)$ be a F\o lner filtration, let $k \colon \N \rightarrow \N$ be a nondecreasing function, and let $\A$ be a $k$-subset of $G$ sampled uniformly at random. Suppose that $k \ll \sqrt[4]{|F_n|}$. Then 
\[ \E[E(\A, \A^{-1})] = \left(2 + \iota(G) \right) k^2 - \left(1 + \iota(G) \right) k + o(1) \,.\]
In particular, when $G$ is abelian, 
\[ \E[E(\A, \A^{-1})] = \left(2 + \frac{1}{|G:I|} \right) k^2 - \left(1 + \frac{1}{|G:I|} \right) k + o(1) \,,\]
where $I$ is the subgroup consisting of involutions.
\end{lemma}
\begin{proof}
Let $A\unlhd G$ be a finite index abelian normal subgroup of $G$, let $T$ be a left transversal for $A$ in $G$, and let $T_0$ be the set of $g\in T$ such that $gA \cap I$ is non-empty for the set $I$ of all involutions. In view of \cref{eq: main AAinv}, we shall prove that the ratio $|Q_3^{(1)}|/|F_n|^2$ tends to $\iota(G)$ as $n$ tends to infinity. For every $g \in T_0$, we define $H_g$ as in \cref{eq:gaga}, so that $gH_g = I \cap gA$. We observe, from the definition of $Q_3^{(1)}$, that
\[
\frac{|Q_3^{(1)}|}{|F_n|^2}
= \frac{1}{|F_n|^2}\sum_{g\in T_0}\# \{(x,y)\in F_n^2 \mid xy^{-1}\in gH_g\} - \frac{1}{|F_n|} \]
Let $ \omega$  be a non-principal ultrafilter, and let $\mu$ be the left-invariant mean arising from $\omega$ together with the original F\o lner sequence. We fix $g\in T$. Observe that $xy^{-1}\in gH_g$ if, and only if, $x\in gH_g y = gyH_g$, where the latter equality following from the fact that $H_g$ is characteristic in $A$, and, thus, normal in $G$. Hence, we compute
\[
\begin{split}
\lim_\omega \frac{1}{|F_n|^2}\#
\bigl\{(x,y)\in F_n^2 \mid xy^{-1}\in gH_g\bigr\}
&= \mu^{\otimes 2}\Bigl(\bigl\{(x,y)\in G^2 \mid xy^{-1}\in gH_g\bigr\}\Bigr) \\
&= \int_G \mu(gyH_g)\, d\mu(y) \\
&= \int_G \mu(H_g) \, d\mu(y) \\
&= \frac{1}{|G:H_g|},
\end{split}
\]
(In the case where $|G:H_g|$ is infinite, the ultimate expression is understood to be zero.)
Furthermore, as the value is independent of the non-principal ultrafilter, the ultralimit is actually a proper limit. Therefore, by \cref{eq:expressionIota},
\[ \lim_n \frac{|Q_3^{(1)}|}{|F_n|^2} =  \sum_{g\in T_0} \lim_n\frac{\left| \{(x,y)\in F_n^2 \mid xy^{-1}\in gH_g\}\right| }{|F_n|^2} = \sum_{g \in T_0} \frac{1}{ |G: H_g|} = \iota(G) \,.\]
The remainder of the proof consists of the usual computations. 
\end{proof}

When the involution probability is zero, and thus, $G$ is not virtually abelian, we expect that the pair $|Q_3^{(1)}|$ and $|Q_3^{(2)}|$ has no influence on the limiting behaviour. In particular, the formula above remains valid with $\iota_\mu(G)=0$.

\begin{lemma}
\label{lem: non virt abelian}
Let $G$ be a residually finite amenable group and let $(F_n) \subseteq G$ be a F\o lner filtration, let $k \colon \N \rightarrow \N$ be a function, and let $\A$ be a $k$-subset of $G$ sampled uniformly at random. Suppose that $G$ is not virtually abelian, and that $k \ll \sqrt[4]{|F_n|}$. Then,
\[ \E[E(\A, \A ^{-1})] = 2 k^2 - k + o(1) \,.\]
\end{lemma}
\begin{proof}
We denote by $\mu_n$ the discrete probability measure associated to $F_n$. The lower bound on $\E[E(\A, \A ^{-1})]$ is straightforward from \cref{thm:multiplicativeInverse}. 
For the upper bound, we shall prove that $|Q_3^{(1)}| = o(|F_n|^2)$. In view of \cref{remark:potatoes}, it suffices for this purpose proving that
\begin{proofequation}
\label{eq: convoluted}
\lim_n \frac1{|F_n|}\sum_{x\in F_n}\frac{|F_n\cap xI|}{|F_n|} = 0 \,,
\end{proofequation}
where $I$ is the set of all involutions. In view of \cref{lemma: involution charac}, $I$ has upper Banach density zero, {\it i.e.},
\[ \sup_{(F_n) \hbox{ F\o lner}} \limsup_{n\to\infty}\frac{|I\cap F_n|}{|F_n|}\ =\ 0.
\]
Hence, by \cite[Proposition~A.6]{upperBanach}, every left-invariant mean  $m$ satisfies $m(\1_I)=0$.

We consider the adjoint measure $\tilde{\mu}_n$, so that, for every $S \subseteq G$, $\tilde \mu_n(S)= \mu_n(S^{-1})$, and the convolution measure $\nu_n =\mu_n*\tilde\mu_n$. (Observe that, if the F\o lner sequence is symmetric, then $\tilde \mu_n = \mu_n$.) 
Note that 
\[
\nu_n(I)
=\frac{1}{|F_n|^2} \#\{(x,y)\in F_n\times F_n \mid \ x^{-1}y\in I\}  
=\frac1{|F_n|^2}\sum_{x\in F_n}|F_n\cap xI|  \,.
\]
Hence, the left-hand side of \cref{eq: convoluted} is $\lim_n \nu_n(I)$.

We aim to prove that $\nu_n$ gives rise to a left-invariant mean. Explicitly, if we define the family of discrete means
\[m_n(f)= \sum_{g\in G} f(g)\nu_n(g), \, \quad \hbox{for every } f \in \ell^{\infty}(G) \,, \]
any ultralimit is left-invariant. By the Banach--Alaoglu Theorem, the set of all means is compact in $\ell^{\infty}(G)^{*}$ with the weak-$*$ topology (see \cite[Theorem 4.2.1]{CS}). Consider an arbitrary weak-$*$ convergent subsequence $(m_{n_k})_{k \in \N}$. We claim that the corresponding weak-$*$ limit is left-invariant.
Let
\[L_g \colon G \rightarrow G , \quad x \mapsto gx\]
be the \emph{left-multiplication by $g$}, and recall that $G$ acts on the set of finitely additive measures of $G$ by the push-forward of $L_g$. In particular,
\[ g \cdot \mu_n (S)= \mu_n (g^{-1} S), \, \quad \, S \subseteq G\,.\]
Furthermore, since $(F_n)$ is a F\o lner sequence,
\[ \lim_n \|g \cdot\mu_n-\mu_n\|_1  = 0,\]
with the $\ell^1(G)$-norm (see \cite[proof of Theorem 4.9.2]{CS}). Therefore, since the convolution by a discrete probability measure is an $\ell^1$-contraction,
\[
\lim_n \|g \cdot\nu_n-\nu_n\|_1
= \lim_n \|(g \cdot\mu_n)*\tilde\mu_n-\mu_n*\tilde\mu_n\|_1
\le \lim_n \|g \cdot\mu_n-\mu_n\|_1  = 0 \,.
\]
Finally, the equivalent left-action of $G$ on the set of all means is given by $g \cdot m (f) = m (f \circ L_{g^{-1}})$, and, from the definition of $m_n$,
\[
\left|(g\cdot m_n-m_n)(f)\right|
\le
\|f\|_\infty \sum_{h\in G}\left|(g\cdot \nu_n)(h)-\nu_n(h)\right|
=
\|f\|_\infty\,\|g\cdot \nu_n-\nu_n\|_1 \,.
\]
In particular,
\[ \lim_n \left|(g\cdot m_n-m_n)(f)\right| = 0 \,.\]
Thus, any weak-$*$ cluster point is a left invariant mean on $\ell^\infty(G)$, proving the claim.

We denote by $m$ the limit point of an arbitrary convergent subsequence $(m_{n_k})$. Since $m$ is a left-invariant mean
\[ \lim_{k} m_{n_k}(\1_I) = m (\1_I) = 0 \,.\]
Hence,
\[ 0= \lim_n m_n(\1_I) = \lim_n \nu_n(I),\]
as we desire.
\end{proof}

The upshot of this section is the following.
\begin{corx}\label{cor:multiplicativeInverseBalls}
    Let $G$ be an amenable infinite group, let $(F_n) \subseteq G$ be a symmetric F\o lner sequence of $G$, let $k \colon \mathbb{N} \to \mathbb{N}$ be a nondecreasing function, and let $\A$ be a $k$-subset of $F_n$ sampled uniformly at random.
    \begin{enumerate}[$(a)$]
        \item Suppose that $k = \alpha |F_n| + o(|F_n|)$, for some constant $\alpha \in [0,1]$. Then,
        \[ \E[\eta_n(\A,\A)] \le \alpha^4 + \alpha^3 \,.\]
        Moreover, if $G$ contains finitely many involutions,
        \[\E[\eta_n(\A,\A)] = \alpha^4 \,.\]
        \item Suppose that $G$ is residually finite and non-virtually abelian. Then 
        $$\E[E_n(\A, \A^{-1})] = 2k^2 - k. $$
        \item Suppose that $G$ is virtually abelian. Then, the involution probability is independent of the F\o lner sequence, and,  
        $$\E[E_n(\A, \A^{-1})] = (2+ \iota(G))k^2 -(1- \iota(G)) k. $$
    \end{enumerate}
\end{corx}
\begin{proof}
    The proof of item~$(a)$ is a straightforward computation from \cref{thm:multiplicativeInverse}. Meanwhile,
    item~$(c)$ follows from \cref{lem: virt abelian}, and item~$(b)$ follows from \cref{lem: non virt abelian}.
\end{proof}
The use of residual finiteness in the proof of \cref{lemma: involution charac} seems a technicality, and it does not have a deep mathematical reason to be there. Hence, in parallel with \cref{quest:lastDilemma}, we pose the following question.
\begin{question}
    Can we drop the hypothesis of residual finiteness in \cref{lemma: involution charac} and \cref{cor:multiplicativeInverseBalls}~$(b)$?
\end{question}

\subsection{Finitely generated groups}\label{sec: 5-inffingen}
Once again, we now shift our attention to finitely generated groups endowed with the word metric and the filtration consisting of balls. To start, we note that we already improved the main result in \cite[Section 2.1]{Tointon}: rather than proving the theorem for the integers alone, \cref{cor:multiplicativeInverseBalls} holds for arbitrary residually finite groups of subexponential growth where the filtration is given by balls (as they form a F\o lner filtration), and $k$ can be chosen to be any increasing function with $k \ll \sqrt[4]{|\B_S(n)}$.

We actually conjecture that this formula does hold for every finitely generated group. 

\begin{conjecture}\label{conj:Jeff}
Let $G$ be an infinite group and let $S$ be a finite generating set for $G$. Then, 
\begin{enumerate}[(i)]
\item $\iota_S(G)$ is positive if, and only if, $G$ is virtually abelian.
\item for every function $k \colon \mathbb{N} \to \mathbb{N}$ such that $k \ll \sqrt[4]{|{\bf B}_S(n)|}$, and for every $\A$, $k$-subset of $G$ sampled uniformly at random in $\B_S(n)$, we have 
$$\E[E(\A, \A^{-1})] = 2  k^2 - k \,. $$
\end{enumerate}
\end{conjecture}

It is quite common for the involution density to vanish in several of the standard classes of groups arising in geometric group theory. For instance, it is the case, independently of the generating set, for word hyperbolic groups (see \cite{Dani}) or generalised lamplighter groups of the form $H \wr \Z$ (see \cite{A3}). We confirm the second part of the conjecture for these groups.

Moreover, it is worth noting that, by virtue of \cite{Sisto}, if rather than considering the word metric, one considers the probability distribution induced by random walks, then the probability of hitting an involution decays exponentially in several other classes.

The first example is essentially trivial. 

\begin{example} Suppose that $G$ has finitely many involutions. 
Let $\A$ be a $k$-subset of $G$ sampled uniformly at random $\B_S(n)$ for any generating set $S$. Then 
$$\E[E_n(\A, \A^{-1}) = 2k^2 - k \,. $$
\end{example}

\begin{example}[Word hyperbolic groups]
\label{ex: AAinv hyperbolic}
Let $G$ be nonelementary word hyperbolic group with finite generating set $S$.
Using \cref{remark:potatoes}, we shall prove that the ratio $|Q_3^{(1)}|/|\B_S(n)|^{2}$ converges to zero as $n$ goes to infinity. We point out the general identity
\begin{equation*}
\begin{split}
        |Q_3^{(1)}|
    &= \sum_{x \in B_S(n)} |\B_S(n) x \cap I|
    \\&= \sum_{y \in I \cap \B_{S}(2n)} \left| \{ x \in \B_S(n) \mid yx^{-1} \in \B_S(n) \} \right| 
    \\&= \sum_{y \in I \cap \B_S(2n)}  \bigl|\B_S(1, n) \cap \B_S(y, n)\bigr| \,,
\end{split}
\end{equation*}
where $\B(z, n)$ denotes the ball of radius $n$ and centre $z$ in the Cayley graph of $G$ with respect to $S$. Moreover, let $|y|_S$ denote the word length of $y$ with respect to the generators in $S$. By \cite[Proposition III.H.1.4]{BrHa}, in every hyperbolic group $G$, there is a constant $C > 0$ such that
\[
\B_S(1,n)\cap \B_S(y,n)\subseteq \B_S\!\left(z,\ n-\frac{|y|_S}{2} + C \right) \,,
\]
for a midpoint $z$ of a geodesic between $1$ and $y$. Hence,
\[
|\B_S(n)\cap y \B_S(n)|\le |\B_S(n-\tfrac{|y|_S}{2} + C )| \,.
\]
Grouping by the length of the involution $y$, we have
\[
|Q_3^{(1)} | \le \sum_{i=0}^{2n} |I \cap \B_S(i)|\cdot |\B_S(n-\tfrac{i}{2}+C )| \,.
\]
By \cite[Theorem 1.1]{Dani}, there exist two positive constants, $d$ and $D$, such that
\[
|I \cap \B_S(i) | \le d |\B_S(\tfrac{i}{2} + D ) | \,.
\]
Therefore,
\[
|Q_3^{(1)}| \le d \sum_{i=0}^{2n} |\B_S(\frac{i}{2} + D)| \cdot |\B_S(n-\tfrac{i}{2}+C)|.
\]
Finally, $G$ has pure exponential growth: more specifically, \cite[Th\'eor\'eme 7.2]{Coo93} yields that there exist two constants $\lambda, c > 1$ such that, for every integer $n \ge 0$,
\[ c^{-1} \lambda^n \leq {\bf B}_S(n) \leq  c \lambda^n \,.\]
Hence,
\[ |\B_S(\tfrac{i}{2}+D )|\cdot |\B(n-\tfrac{i}{2}+C)|
\le c^2 \lambda^{n+D+C} \,,
\]
and, thus, $|Q_3^{(1)}| \leq M n \lambda^n$, for a large enough constant $M \geq 0$. Finally, since $|\B_S(n)| \geq c^{-1}\lambda^n$, we compute
\[ \frac{|Q_3^{(1)}|}{|\B_S(n)|^2}\le \frac{M n \lambda^n}{c^{-2}  \lambda^{2n}} = \Theta( n \lambda^{-n}) \,.
\]
In particular, the ratio $|Q_3^{(1)}| / |\B_S(n)|^2$ converges to zero as $n$ tends to infinity. This is enough to conclude as in \cref{lem: virt abelian}, thus proving the veracity of \cref{conj:Jeff} in this context.
\end{example}

\begin{example}[Generalised lamplighter groups]
\label{ex: AAinv lamp}
Let $G= H \wr \Z$ be a generalised lamplighter group, and let $N= H^\Z$ be its base subgroup.  Suppose that $G$ has {\it pinched exponential growth} with respect to $S$, {\it i.e.}, for some constants $c, C > 0$,
\[ c\lambda^n \leq |\B_S(n)| \leq C \lambda^n \,,\]
where
\[\lambda = \lim_n \sqrt[n]{|\B_S(n)|}\]
is the \emph{exponential growth ratio} of $G$ with respect to $S$. For instance, this is the case when $H$ is finite and the generating set $S$ contains the cyclic generator of $\Z$. (This can be seen by the growth series using \cite{Johnson}. We also refer to \cite[Lemma 3.7]{BG} for an elementary proof.) From the fact that every involution is contained in $N$, we aim to derive that $|Q_3^{(1)}|/|\B_S(n)|^2$ converges to zero as $n$ tends to infinity.

By \cite[Theorem B]{A3},
\[\lim_n \frac{|\B_S(n) \cap N |}{|\B_S(n)|} = 0 \,. \]
In particular, by its proof, the decay ratio is  
\[ \frac{|N \cap \B_S(n)|}{|\B_S(n)|} \in  o\left(\frac{1}{\log(n)}\ \right)\]
independently of the generating set. Observe that, for every $g \in G$, there is a constant $m>0$ such that 
\[ \frac{|\B_S(n) \cap Ng|}{|\B_S(n)|} \leq \frac{|\B_S(n+ |g|_S) \cap N|}{|\B_S(n+ |g|_S)|} \frac{|\B_S(n+ |g|_S)|}{|\B_S(n)|} \leq \frac{m \lambda^{|g|_S}}{\log (n + |g|_S)} \,.\]
(We use the notation $|g|_S$ for the word length of $g$ with respect to the generating set $S$.)

Therefore, since $I$ is contained in the base group $N$, we compute 
\[ \begin{split}
\lim_n \frac{|Q_3^{(1)}|}{|\B_S(n)|^2}
&= \lim_n \frac{1}{|\B_S(n)|} \sum_{x \in \B_S(n)} \frac{|\B_S(n) \cap Ix|}{|\B_S(n)|} \\
&\le \lim_n \frac{1}{|\B_S(n)|} \sum_{x \in \B_S(n)} \frac{|\B_S(n) \cap N x|}{|\B_S(n)|} \\
&\le \lim_n \frac{m}{\lambda^n}\sum_{i=0}^n \frac{\lambda^{i}}{\log(n+i)}
\\&= 0 \,.
\end{split} \]
Once again, this concludes the proof of \cref{conj:Jeff} for this case.
\end{example}

\section{Subset of large growth}\label{sec:growth}
One of our applications of the estimates obtained so far is to employ \cref{lemma:CauchySchwartz} together with probabilistic methods to identify subsets of discrete groups exhibiting large growth.
This line of research is currently very active, particularly in the setting of finite simple groups, where recent advances include \cite{DonaMarotiPyber2024,KeevashLifshitz2023,Skresanov2025}.
For amenable locally compact groups, related questions have been investigated from an ergodic and measure-theoretic perspective in \cite{Austin, upperBanach}. Analogous problems have also been explored in the continuous setting: for instance, the recent breakthroughs \cite{JingTranZhang2023, Machado2024} establish sharp measure-doubling results for small subsets of compact Lie groups.

Preliminarily, we prove a lemma that we will use in most of the proofs in this section.
\begin{lemma}\label{lem:largedoublinglemma}
    Let $h: \NN \rightarrow \NN$ be an unbounded function. Then, there exists a sequence $\delta_n$ converging to $1$, and a positive integer $M(h)$ such that, for every discrete group $G$ endowed with an increasing filtration series $(F_n)$, and for every integer $n$ such that $|F_n| \ge M(h)$, there exists a subset $A_n \subseteq F_n$ such that
    \[|A_n| \le |F_n|^{\frac{1}{2}}h(|F_n|)\quad \hbox{and} \quad |A_n^{*2}| \ge \frac{|F_n|}{\delta_n} \,.\]
\end{lemma}
\begin{proof} Preliminarily, note that we can assume that $h(n) \le \sqrt[4]{n}$, because, if the lemma holds for a function $h$, then the statement holds with the same $\delta_n$ and $M(h)$ for every greater function. 

We define, for some $\alpha>0$, the sequence
\[ \delta_n = 1 + \frac{3+\alpha}{h(n)^{2}}
\,,\]
and the function
\[ k(n) =  h(n) \sqrt{n}\,.\]
We fix an arbitrary discrete group $G$ and its increasing filtration series $(F_n)$, and we consider a random $k(|F_n|)$-subset $\A$ of $F_n$. 
By \cref{thm:multiplicative}, as $n$ tends to infinity,
\[\frac{|F_n|}{k^4} \E\left[ E(\A, \A) \right] \le 1 + \frac{3}{h^{2}} + o\left(|F_n|^{-1/2} \right)\,.\]
By Markov's inequality,
 \[ \P\left[ E(\A,\A) \ge \frac{\delta_n k^4}{|F_n|} \right] \le \frac{|F_n|}{\delta_n k^4} \E\left[ E(\A,\A) \right] \,.\]
 By taking the complementary event,
    \[ \P\left[ E(\A,\A) \le \frac{\delta_n k^4}{|F_n|} \right] \ge 
    1 - \frac{|F_n| \, \E\left[ E(\A, \A) \right]}{\delta_n k^4} \ge \frac{1 + \frac{3}{h^{2}} + o\left(|F_n|^{-1/2} \right)}{ 1 + \frac{3+\alpha}{h^{2}}} \,.\] Indeed, by the choice of $\delta_n$, there exists a positive integer $N$ such that, if $n\ge M$, this event has positive probability. 
    In particular, there exists a $k$-subset $A_n$ of $F_n$ such that
    \[ E(A_n, A_n)  \le \frac{\delta_n k^4}{|F_n|} = \delta_n h^{4} |F_n| \,. \]
    Therefore, \cref{lemma:CauchySchwartz} implies
    \[ |A_n^{*2}| \ge \frac{|A_n|^4}{E(A_n, A_n)} \ge \frac{h^4 |F_n|^{2} }{\delta_n h^4 |F_n|} = \frac{|F_n|}{\delta_n} \,,\]
    as desired
\end{proof}

\begin{remark}
    Observe that the same proof applies \emph{verbatim} provided that we replace the second copy of $\A$ with $\A^{-1}$ and we included the additional assumption that $F_n$ is symmetric. Indeed, all the results of \cref{sec:growth} can be reproved for the pair $(\A,\A^{-1})$. For the sake of brevity, and in view of the aesthetic preference for working with the pair $(\A,\A)$, we omit the details.
\end{remark}

\subsection{Finite groups}\label{sec:diffBass-finite}
An \emph{additive basis} of a group $G$ is defined as a subset $A \subseteq G$ with the property that $G=A^{\ast 2}$. It is known that every finite group $G$ admits an additive basis $A$ such that $|A| = O(|G|^{1/2}),$ see \cite{KozmaLev1992}. However, the proof relies on the Classification of Finite Simple Groups. Moreover, this notion has found application both within mathematics \cite{AlonBukhSudakov2007,BarbieriLeksePotocnikRekvenyi2026,RedeiRenyi1949} and outside in computer science \cite{CloteKranakis1989,EvraGadotKleinKomargodski2024,FinkelsteinKleitmanLeighton1988}.

\cref{thm:largeDoublingFinite} morally states that, asymptotically, there exists a set which has cardinality slightly larger than $O(|G|^{1/2})$ which is almost an additive basis. The novelty of our proof is that it does not rely on the Classification of Finite Simple Groups.
\begin{thmx}\label{thm:largeDoublingFinite}
    For every $\epsilon > 0$, and for every unbounded function $h\colon \mathbb{N} \to \mathbb{N}$, there exists a positive integer $N = N(h, \epsilon)$ such that every finite group $G$ with $|G|\ge N$ contains a subset $A$ such that
    \[ |A| \le |G|^{\frac{1}{2}}h(|G|), \quad \hbox{and} \quad |A^{*2}| \ge (1-\epsilon)|G| \,.\]
\end{thmx}

\begin{proof}
    Let $\delta_n$ and $M(h)$ be the sequence and the integer described in \cref{lem:largedoublinglemma}.
    Since $G$ is a finite group, we can take as filtration series $(F_n)$ where $F_i = G$ for all $i$. If $|G| > M(h)$, by \cref{lem:largedoublinglemma}, there exists a subset $A$ such that \[ |A| = |G|^{\frac{1}{2}}h(|G|), \quad \hbox{and} \quad |A^{*2}| \ge \frac{|G|}{\delta_n} \,.\]
    Since $\delta_n$ converges to $1$, there exists an integer $N(h,\epsilon) \ge M(h)$ such that, for every $n \ge N(h,\epsilon)$, $(1-\epsilon)\le 1/ \delta_n$. This completes the proof.
\end{proof}

We can leverage this result to generate the whole group in a bounded number of steps,
provided that the algebraic structure of $G$ does not trap the set $A$ inside a proper
subgroup. For instance, let $G$ be a $p$-group of order $p^e$, and let $M$ be a maximal subgroup of $G$ of order $p^{e-1}$. A $k$-random subset $\A$ of $G$ has positive probability of being contained in $M$, namely,
\[ \P[\A \subseteq M] = \frac{{p^{e-1} \choose k}}{{p^e \choose k}} = \frac{1}{p^k} + o(1) \,,\]
where the last asymptotic is taken as $e$ goes to infinity. Hence, when $A \subseteq M$, $A^{*2}$ cannot generate the whole group. To rule out this type of obstruction, we
impose a mild representation-theoretic condition on $G$. Namely, we assume that $G$ is
\emph{$\epsilon$-quasirandom}, that is, every nontrivial linear representation of $G$
has dimension at least $|G|^\epsilon$. (This notion was introduced in \cite{Gowers2008}.
An algebraic classification of quasirandom groups has been provided in
\cite{BarbieriSabatini2025}.)

\begin{corx}\label{cor:GowersTrick}
    For every $\epsilon > 0$, and for every function $h: \mathbb{N} \to \mathbb{N}$ with $1 \ll h(n)$, there exist a positive integer $N = N(h, \epsilon)$ such that every $\epsilon$-quasirandom group $G$ with $|G|\ge N$ contains a subset $A$ such that
    \[ |A| \le |G|^{\frac{1}{2}}h(|G|), \quad \hbox{and} \quad A^{*6}  = G \,.\]
\end{corx}
\begin{proof}
    By eventually increasing $N$, observe that
    \[ (1 - \epsilon)|G| \ge |G|^{\frac{3-\epsilon}{3}} \,.\]
    We can now apply Gowers's trick (see \cite[Corollary 1]{NikolovPyber2011}) with all the sets involved equal to $A^{*2}$.
\end{proof}

\subsection{Infinite groups} \label{subsec: infite growth}
We start by giving an explicit example, in the infinite cyclic group $\Z$, of a set of measure zero that has substantial growth after a single product. This is a simple instance of a well-known phenomenon in additive number theory, known as \emph{thin bases}. They are subsets $A$ of the natural numbers with density zero whose $k$-fold sumsets nevertheless contains any large enough integer, and hence has density $1$. (See \cite{base1, base3} for classical references.)

\begin{example}
    Let $A$ be the set of squares, that is
    \[ A=\{ \pm m^2 \mid m\in\ZZ\}\subseteq \ZZ \,.\]
    For every positive integer $n$, define $A_n=A\cap [-n,n]$. On one hand, the number of squares in $[-n,n]$ is $|A_n| = 2\lfloor \sqrt{n}\rfloor + 1$, and, thus,
    \[ \mu(A) = \lim_n \frac{|A_n|}{2n+1} = 0 \,. \]
    On the other hand, an integer $m$ is a difference of two squares if, and only if, $m\not\equiv 2\pmod 4$.
    Hence,
    \[ \mu(A+A) = \lim_n \frac{|A_n+A_n|}{2n+1} = \frac{3}{4} \,.\]
    Therefore, though $A$ is a set of measure zero, its sumset has positive measure.
\end{example}

We believe that this behaviour is actually a general feature of amenable groups.

\begin{question}
Let $G$ be an amenable group. Does there exist a
subset $A \subseteq G$ such that $ \mu(A) = 0$  and $\mu(A^{*2}) > 0 $?
\end{question}

In this subsection, we answer this question in the affirmative for directed limits of finite groups with specified measures. This phenomenon highlights a failure of subadditivity of measure under group multiplication in amenable groups. 

First, we note that any subsequence of a F\o lner sequence $(F_n)$ is also a F\o lner sequence. In order to define the finitely additive measure $\mu$ from  $(F_n)$ using a non-principal ultrafilter $\omega$ on $\N$, {\it  à la} \cref{eq: define a measure with ultrafilter}, we may pass to a subsequence without changing the values of $\mu$, provided we modify the ultrafilter accordingly.

Explicitly, let $f \colon \N \rightarrow \N$ be the increasing function corresponding to the subsequence $(F_{f(n)})$. Consider the nonprincipal ultrafilter $f^{-1}(\omega)$ on $\N$ defined by
\[
f^{-1}(\omega) = \{S\subseteq \N \mid f(S) \in\omega \}.
\]
Then, for every $A\subseteq G$, we have
\[
\mu(A)= \lim_{\omega}\frac{|A\cap F_n|}{|F_n|}
=\lim_{f^{-1}(\omega)}\frac{|A\cap F_{f(n)}|}{|F_{f(n)}|}.
\]
Therefore, throughout the sequel, we shall freely pass to subsequences of F\o lner sequences, tacitly replacing the ultrafilter by the corresponding pullback, so that the resulting measure $\mu$ remains unchanged.

In particular, when $G$ is an infinite amenable group with F\o lner sequence $(F_n)$, we can assume that the sequence
\[\alpha_n = \frac{|F_{n-1}|}{|F_n|}\]
of rational numbers converges to zero, and that $\alpha_n < 1$ for every $n$.
Moreover,  we can use this to get F\o lner sequences with arbitrarily fast decay ratio. For any decreasing sequence of real numbers $\beta_i$, since the sequence $|F_n|$ is unbounded (see \cite[Exercise 4.52]{CS}), there exists an increasing sequence of integers $(n_i)$ such that 
\[ |F_{n_i}| \geq \frac{|F_{n_{i-1}}|}{\beta_i} \,.\]
Hence, the subsequence $(F_{n_i})$ has decay ratio $O(\beta_i)$.

We need the following technical lemma to obtain F\o lner sequences where all the subsets are pairwise disjoint. 

\begin{lemma}\label{lemma:MarusaRevenge}
    Let $G$ be an infinite amenable group, and let $(G_n)_{n \in \N}$ is a nested F\o lner sequence. There exists a F\o lner sequence $(F_n)_{n \in \N}$ of $G$ such that the sets of the sequence are pairwise disjoint.
\end{lemma}
\begin{proof}
As explained before the proof, up to passing to a subsequence, we can assume that the sequence
\[\alpha_n = \frac{|G_{n-1}|}{|G_n|}\]
converges to zero, and that $\alpha_n <1$ for every $n$. 
Let $(F_n)$ be the sequence of finite sets defined by 
\[ F_n = G_n \setminus  G_{n-1} \,.\]
Note that the sets are pairwise disjoint by construction. We aim to prove that $(F_n)$ is actually F\o lner. Indeed, for every $g\in G$, we compute
\begin{align*}
    \frac{\lvert F_n g \triangle F_n\rvert}{\lvert F_n\rvert}
    &= \frac{\left\lvert (G_n\setminus G_{n-1})g \triangle (G_n\setminus G_{n-1}) \right\rvert}{\left\lvert G_n\setminus G_{n-1}\right\rvert} \\
    &\le \frac{\lvert G_n \triangle G_n g\rvert + 2\lvert G_{n-1}\rvert}{(1-\alpha_n)\lvert G_n\rvert} \\
    &= \frac{\lvert G_n \triangle G_n g\rvert}{(1-\alpha_n)\lvert G_n\rvert}
    + \frac{2\alpha_n}{1-\alpha_n}\,.
\end{align*}
The last term converges to zero, because $(G_n)$ is a F\o lner sequence, and $\alpha_n$ converges to zero.
\end{proof}

Recall that infinite direct limits of finite groups
\[G = \varinjlim G_n\]
are precisely countably infinite locally finite groups. Moreover, they are amenable, since $(G_n)$ constitutes a nested F\o lner sequence consisting of finite subgroups. In the next theorem, we consider the left invariant measure $\mu$ corresponding to this F\o lner sequence. 

 \begin{thmx}\label{thm:largeDoubling}
    Let $G$ be a countably infinite locally finite group. Then, there exists a subset $A \subseteq G$ such that
    \[\mu(A) = 0, \quad \hbox{and} \quad \mu(A^{\ast 2}) = 1 \,.\]
\end{thmx}

\begin{proof}
 Let $(G_n)$ be a F\o lner sequence made of subgroups of $G$ defining $\mu$ together with a non-principal ultrafilter $\omega$ in $\N$, and let $(F_n)$ be a F\o lner sequence obtained from $G_n$ as described in \cref{lemma:MarusaRevenge}. 
 Without loss of generality, as explained before the statement of \cref{lemma:MarusaRevenge}, we assume that
 \[\beta_n = |F_{n-1}|/|F_n| \ll 1/n \,\]
 and that, for all positive integers $n$, $\beta_n < 1$. Finally, let $1\ll h < \sqrt[4]{n}$ be a function.
 
 By \cref{lem:largedoublinglemma}, there exist a sequence $\delta_n$ converging to $1$ and a positive integer $M$ such that, if $|F_n|\ge M$, then there exists a subset $A_n \subseteq F_n$ such that
 \[|A_n| \le |F_n|^{\frac{1}{2}}h(|F_n|)\quad \hbox{and} \quad |A_n^{*2}| \ge \frac{|F_n|}{\delta_n} \,.\] 
We define
    \[ A = \bigcup_{n=M}^\infty A_n \,.\]
    On the one hand,
    \[\begin{split}
        \mu(A) &= \lim_\omega \frac{\left| A \cap G_n   \right|}{\left| G_n\right|}
        \\&= \lim_\omega \frac{\left| A \cap \left( \bigcup_{j=0}^{n} F_j \right) \right|}{\left| \bigcup_{j=0}^{n} F_j\right|}
        \\&\le \lim_\omega \frac{\sum_{j=M}^n |A_j|}{\sum_{j=0}^{n} |F_j|}
       \\ &\le \lim_\omega \frac{\sum_{j=M}^n |F_j|^{\frac{1}{2}}h(|F_j))}{\sum_{j=0}^{n} |F_j|}
        \\&= 0 \,.
    \end{split}\]
    On the other hand, as each $G_n$ is a subgroup, $G_n=G_n^{*2}$. Thus, for every $j\leq  n$,
    \[ F_j^{*2} \subseteq G_j^{*2} = G_j \subseteq G_n \,.\]
    Hence, we compute
    \[\begin{split}
        \mu(A^{*2}) &= \lim_\omega \frac{\left| A^{*2} \cap G_n \right|}{\left| G_n \right|} \\& 
        \geq \lim_\omega \frac{\left| A^{*2} \cap \left( \bigcup_{j=0}^{n} F_j^{* 2} \right) \right|}{\left| \bigcup_{j=0}^{n} F_j \right|}
        \\&\ge \lim_\omega \frac{|A_n^{*2}|}{\sum_{j=0}^{n} |F_j|}
        \\&\ge \lim_\omega \frac{ \delta_n^{-1}|F_n|}{|F_n|(1 + n\beta_n)}
        \\&= \lim_\omega \frac{1}{\delta_n(1 + n\beta_n)} 
        \\&= 1 \,.
    \end{split}\]
    Therefore, $A$ has the desired properties.
\end{proof}

Observe that successive products of the sets $A$ provided by the previous theorem cannot remain trapped in a proper subgroup, which has measure at most $1/2$, and, hence, $A$ must generate $G$. It is then natural to ask what is the \emph{width} of $G$ with respect to the subset $A$.

\begin{question}
    How many additional steps $k$ are required to ensure that $A^{\ast (2+k)} = G$? (The value of $k$ might as well be infinite.)
\end{question}

\section{Sum or difference domination} \label{sec:sumDiffDom}Let $A$ be a subset of the integers $\mathbb{Z}$. We say that $A$ is \emph{sum-dominant}
if $|A+A|>|A-A|$, \emph{difference-dominant} if $|A+A|<|A-A|$, and \emph{balanced}
if $|A+A|=|A-A|$. Examples of these behaviours are given, for instance, in
\cite{Ruzsa1992}. Since addition is commutative but subtraction is not, one might
expect that the vast majority of sets are difference-dominant. This is false, in the
sense that, if we sample a subset $\Ainf$ of $\{1,2,\dots,n\}$ by picking every
integer with probability $1/2$, then there is a positive probability that $\Ainf$ is
sum-dominant (see~\cite{GregKevin2007,Zhao2011}). The analysis involved in proving relies on the fact that, in the bulk of the interval, sumsets and difference sets behave roughly the same, while what really matters is how the elements in the fringe of $\Ainf$
are distributed. As a consequence, one might expect balanced sets to appear more
often if we change the domain of the sampling to a finite group $G$, where the concept of fringe of an interval loses its meaning. Indeed, the fact that balance subsets appear with high probability is proved in \cite{MillerKevin2014}.

In this section, we tackle a similar question -- where we allow the ratio to be a
fixed constant -- using our energy estimates and \cref{lemma:CauchySchwartz}.
Furthermore, observe that our probability space is widely different from those used
in the works mentioned above. (We also point out that recent works
\cite{ChuVidt2022,HemmadyLottMiller2017} comparing difference-dominance and
sum-dominance analyse the situation by changing the sampling strategy, by using distinct Bernoulli random variables for each point.) Indeed, in
our case we consider a genuinely sparse setting, where $k\ll \sqrt{|F_n|}$, whereas
in the previous model the expected cardinality of $\A$ is linear in $|F_n|$.

\begin{thmx}\label{thm:sumDiffDom}
    Let $G$ be an infinite discrete group, let $(F_n)$ be a symmetric filtration series of $G$, let $k: \mathbb{N} \to \mathbb{N}$ be an unbounded function with $k \ll \sqrt{|F_n|}$, and let $\delta,\epsilon \in (0, 1/3)$,  $c \in (0,1)$ be three constants. Then, there exists an integer $N = N(k, \delta, c)$ such that, for every integer $n\ge N$,
    \[ \P \left[ \frac{|\A\A^{-1}|}{|\A^{\ast 2}|} \le \left( \frac{1}{3}-\delta \right)^{-1} \right] > c \delta \,,\]
    and
    \[ \P \left[ \frac{|\A^{\ast 2}|}{|\A\A^{-1}|} \le \left( \frac{1}{3}-\epsilon \right)^{-1} \right] > c \epsilon \,.\]
\end{thmx}
\begin{proof}
Recall that, by applying \cref{lemma:CauchySchwartz} together with Jensen's inequality, we obtain
\[
\E[|\A^{\ast 2}|] \ge \frac{k^4}{\E[E(\A,\A)]},
\quad \hbox{and} \quad
\E[|\A\A^{-1}|] \ge \frac{k^4}{\E[E(\A,\A^{-1})]} \,.
\]
In view of Theorems~\ref{thm:multiplicative} and~\ref{thm:multiplicativeInverse},
\begin{proofequation}\label{eq:qui}
    \E[E(\A,\A)] \le 3 k^2 \left( 1+o(1) \right),
    \quad \hbox{and} \quad
    \E[E(\A,\A^{-1})] \le 3 k^2 \left( 1+o(1) \right) \,.
\end{proofequation}
This yields
\begin{proofequation}\label{eq:Casablanca}
\E[|\A^{\ast 2}|] \ge \frac{k^2}{3+o(1)},
\quad \hbox{and} \quad 
\E[|\A\A^{-1}|] \ge \frac{k^2}{3+o(1)} \,.
\end{proofequation}
Consider the two complementary events
\[
X_\delta = \left\{ |\A^{\ast 2}| \ge \left( \frac{1}{3}-\delta \right)|\A\A^{-1}| \right\},
\quad \hbox{and} \quad
Y_\delta = \left\{ |\A^{\ast 2}| <  \left( \frac{1}{3}-\delta \right) |\A\A^{-1}| \right\} \,.
\]
Since both random variables $|\A^{\ast 2}|$ and $|\A\A^{-1}|$ take values in the interval $[k,k^2]$, we have
\begin{align*}
\E\left[\left( |\A^{\ast 2}| - \left( \frac{1}{3}-\delta \right) |\A\A^{-1}| \right)\mathbf{1}_{X_\delta} \right] & \le \left( k^2 - \left( \frac{1}{3}-\delta \right) k\right) \P[X_\delta] \,,\\
\E \left[ \left( |\A^{\ast 2}| - \left( \frac{1}{3}-\delta \right) |\A\A^{-1}| \right) \mathbf{1}_{Y_\delta} \right] &\le 0 \,.
\end{align*}
Consequently,
\[
\E \left[ |\A^{\ast 2}|- \left( \frac{1}{3}-\delta \right)|\A\A^{-1}| \right] \le \left( k^2 - \left( \frac{1}{3}-\delta \right) k\right) \P[X_\delta] \,.
\]
Combining this with \cref{eq:Casablanca}, we obtain
\[ \begin{split}
    \P[X_\delta]
    &\ge \frac{ \E[ |\A^{\ast 2}|]- \left( \frac{1}{3}-\delta \right) \E[|\A\A^{-1}|]}{k^2-\left( \frac{1}{3}-\delta \right)k}\\
    &\ge \frac{k^2 \left( \frac{1}{3 + o(1)} - \frac{1}{3} + \delta \right)}{k^2(1-o(1))} \,.
\end{split}
\]
This completes the proof of the first statement.

To prove the second, we repeat the same reasoning with the complementary events
\[
X_\epsilon = \left\{ |\A\A^{-1}| \ge \left( \frac{1}{3}-\epsilon \right)|\A^{\ast 2}| \right\},
\quad \hbox{and} \quad
Y_\epsilon = \left\{ |\A\A^{-1}| < \left( \frac{1}{3}-\epsilon \right)|\A^{\ast 2}| \right\} \,. \qedhere
\]
\end{proof}

The proof of \cref{thm:sumDiffDom} can be readily adapted once extra information on the group is known by, rather than using Theorems~\ref{thm:multiplicative} and~\ref{thm:multiplicativeInverse}, evoking one of their corollaries for the appropriate infinite case.

Furthermore, an analogous result can be given for sequences of finite groups.
\begin{thmx}\label{thm:sumDiffDomFinite}
    Let $(G_n)$ be a sequence of finite groups, and define
    \[ \varepsilon = \limsup_n \frac{\varepsilon(G_n)}{|G_n|}, \quad \kappa = \limsup_n \frac{\kappa(G_n)}{|G_n|}, \quad \iota = \limsup_n \frac{\iota(G_n)}{|G_n|} \,. \]
    Let $k$ be a positive integer, and let $\delta \in (0, (1 + \varepsilon + \kappa)^{-1})$, $\epsilon \in (0, (2+ \iota)^{-1})$, $c \in (0,1)$ be three constants. Then, there exists an integer $N = N(k, \delta, \epsilon,c)$ such that, for every integer $n\ge N$,
    \[ \P \left[ \frac{|\A\A^{-1}|}{|\A^{\ast 2}|} \le \left( \frac{1}{1+\varepsilon+\kappa}-\delta \right)^{-1} \right] > c \delta \,,\]
    and
    \[ \P \left[ \frac{|\A^{\ast 2}|}{|\A\A^{-1}|} \le \left( \frac{1}{2 + \iota}-\epsilon \right)^{-1} \right] > c \epsilon \,.\]
\end{thmx}
\begin{proof}
    Using Corollaries \ref{cor:multiplicativeFinite} and \ref{cor:multiplicativeInverseFinite}, we can specialize \cref{eq:qui} to
    \[ \E[E(\A,\A)] \le (1+\varepsilon + \kappa) k^2 \left( 1+o(1) \right),
    \quad \hbox{and} \quad
    \E[E(\A,\A^{-1})] \le (2 + \iota) k^2 \left( 1+o(1) \right) \,.\]
    The rest of the proof is identical to that of \cref{thm:sumDiffDom}.
\end{proof}
For instance, if $(G_n)$ is chosen to be a sequence of nonabelian simple groups, the first portion of the statement of \cref{thm:sumDiffDomFinite} becomes, for every $\delta, c \in (0,1)$, 
\[ \P \left[ \frac{|\A\A^{-1}|}{|\A^{\ast 2}|} \le \left( 1-\delta \right)^{-1} \right] > c \delta \,.\]

On the other hand, we believe that the bounds obtained in Theorems \ref{thm:sumDiffDom} and \ref{thm:sumDiffDomFinite} are far from optimal.
Indeed, our argument relies exclusively on lower bounds on expected values of the relevant random variables, which are inherently insensitive to
fluctuations or fringe phenomena. As a consequence, while we can show that neither $|\A^{\ast 2}|$ nor
$|\A\A^{-1}|$ typically dominates the other by more than a constant factor, our method
does not allow us to tackle the problem with a constant arbitrarily close to $1$.

\begin{question}\label{question:higherMoments}
    What can be said about the moments of order $h\ge 2$ of the random variables $E(\A,\A)$ and $E(\A,\A^{-1})$? In particular, can we obtain from such information estimates for the moments of $|\A\A^{-1}|$ and $|\A^{\ast 2}|$?
\end{question}

\bibliographystyle{plain}
\bibliography{biblio}
\end{document}